\documentclass{amsart}
\usepackage{amssymb}
\newtheorem{theorem}{Theorem}[section]
\newtheorem{definition}[theorem]{Definition}

\newtheorem{lemma}[theorem]{Lemma}
\newtheorem{proposition}[theorem]{Proposition}
\newtheorem{rem}[theorem]{Remark}
\numberwithin{equation}{section}

\newcommand{\Q}{\mathbb Q}

\newcommand{\F}{\mathbb F}
\newcommand{\R}{\mathbb R}
\newcommand{\pp}{\mathcal{P}}

\newcommand{\Z}{\mathbb Z}
\newcommand{\ZK}{{\mathbb Z}_K}
\newcommand{\p}{\mathfrak p}

\newcommand{\q}{\mathfrak q}

\def\as#1{\renewcommand\arraystretch{#1}}

\def\be{\bigskip}

\def\disc{\operatorname{disc}}

\def\fp{{\mathbb F}_p}
\def\fph{\F_{\phi}}

\def\gen#1{\big\langle\, {#1} \,\big\rangle}
\def\imp{\,\Longrightarrow\,}
\def\id{\ind_{\phi}^{(2)}}
\def\ind{\operatorname{ind}}

\def\la{\lambda}
\def\leg#1#2{\as{.8}\left(\begin{array}{c}#1\\\hline#2
         \end{array}\right)}
\def\lra{\longrightarrow}

\def\md#1{\ \mbox{\rm(mod }{#1})}
\def\om{\omega}
\def\n{\op{N}^{(2)}_{\phi}}
\def\nph{\op{N}_{\phi}}
\def\npp{\op{N}_{\phi}^-}

\def\op{\operatorname}

\def\om{\omega}
\def\ord{\operatorname{ord}}
\def\p{\mathfrak{p}}

\def\pp{\mathcal{P}}

\def\qb{\overline{\Q}}

\def\rd{\operatorname{red}}

\def\sii{\ \Longleftrightarrow\ }

\def\t{\theta}

\def\tq{\,\,|\,\,}
\def\vp{v_p^{(2)}}
\def\v2{v_2^{(2)}}

\def\zp{\Z_{(p)}}
\def\zpx{\Z_p[x]}
\title{Newton polygons and $p$-integral bases}

\newfont{\tit}{cmr12 scaled \magstep3}

\author[El fadil]{Lhoussain El fadil}
\address{Lyc\'ee Reda Slaoui (CPGE), 
P.O. Box 3149, 
Talborjt, Agadir, Morocco}
\email{lhoussainelfadil@yahoo.fr}

\author[Montes]{\hbox{Jes\'us Montes}}
\address{Departament de Ci\`encies Econ\`omiques i Socials,
Facultat de Ci\`encies Socials,
Universitat Abat Oliba CEU,
Bellesguard 30, E-08022 Barcelona, Catalonia, Spain}
\email{montes3@uao.es}

\author[Nart]{\hbox{Enric Nart}}
\address{Departament de Matem\`{a}tiques,
         Universitat Aut\`{o}noma de Barcelona,
         Edifici C, E-08193 Bellaterra, Barcelona, Catalonia, Spain}
\email{nart@mat.uab.cat}
\thanks{Partially supported by MTM2006-11391 from the Spanish MEC}
\date{}

\begin{document}

\maketitle

\begin{abstract}
Let $p$ be a prime number. In this paper we use an old technique of \O. Ore, based on Newton polygons, to construct in an efficient way $p$-integral bases of number fields defined by a {\it $p$-regular} equation. To illustrate the potential applications of this construction, we show how this result yields a computation of a $p$-integral basis of an arbitrary quartic field in terms of a defining equation.   
\end{abstract}

\section*{Introduction}
In his 1923 PhD thesis and a series of subsequent papers, \O ystein Ore extended the arithmetic applications of Newton polygons \cite{ore0}, \cite{ore2}. Let $f(x)\in\Z[x]$ be a monic irreducible polynomial, $K$ the number field generated by a root $\t$ of $f(x)$, and $p$ a prime number. If $f(x)$ is \emph{$p$-regular} (Definition \ref{pregular}), Ore determined the prime ideal decomposition of $p$ in $K$ and the $p$-valuation of the index of $\Z[\t]$ inside the maximal order, in terms of combinatorial data attached to different $\phi$-Newton polygons, where the polynomials $\phi(x)$ are monic lifts to $\Z[x]$ of the different irreducible factors of $f(x)$ modulo $p$. In this paper we show that the ideas of Ore can be extended to efficiently compute a $p$-integral basis too. Actually, the data that is required to build a $p$-integral basis is a by-product of the $\phi$-adic developments of $f(x)$, whose computation is necessary to build up the Newton polygons.

In section \ref{secNP} we briefly recall the work of Ore, and in section \ref{secMain} we prove our main theorem (Theorem \ref{main}). To illustrate its wide range of applicability, we use it to obtain an explicit $p$-integral basis of an arbitrary quartic number field, in terms of a defining equation; to this end are devoted sections \ref{secPrel} and \ref{secComputation}. 
A few cases that cannot be solved by Theorem \ref{main} are handled in section \ref{secOrder2} after generalizing this theorem to \emph{second order Newton polygons} (Theorem \ref{main2}).   

The theory of Newton polygons of higher order was developed in \cite{m} (and revised in \cite{GMN}) as a tool to factorize separable polynomials in $\Z_p[x]$. Theorems \ref{main} and \ref{main2} have to be considered as the first steps towards a fast algorithm to compute integral basis in number fields, based on these higher order Newton polygons \cite{GMN2}. 

\section{Newton polygons}\label{secNP}
\subsection{$\phi$-Newton polygons}
Let $p$ be a prime number. Let $\Z_p$ be the ring of $p$-adic numbers, $\Q_p$ the fraction field of $\Z_p$ and $\qb_p$ an algebraic closure of $\Q_p$. 
We denote by $v_p\colon \qb_p^{\,*}\longrightarrow \Q$, the $p$-adic valuation normalized by $v_p(p)=1$; we extend $v_p$ to a discrete valuation of $\Q_p(x)$ by letting it act in the following way on polynomials:
$$
v_p(a_0+a_1x+\cdots+a_rx^r):=\min_{0\le i\le r}\{v_p(a_i)\},
$$

Let $\phi(x)\in\Z_p[x]$ be a monic polynomial of degree $m$ whose reduction modu\-lo $p$ is irreducible. We denote by $\fph$ the finite field $\zpx/(p,\phi(x))$, and by 
$$\hphantom{m}^{\overline{\hphantom{m}}}\,\colon\zpx\lra\fp[x],\qquad \rd\colon \zpx\lra \fph$$
the respective homomorphisms of reduction modulo $p$ and modulo $(p,\phi(x))$.

Any $f(x)\in\Z_p[x]$ admits a unique $\phi$-adic development:
\begin{equation}\label{phiadic}
f(x)=a_0(x)+a_1(x)\phi(x)+\cdots+a_r(x)\phi(x)^r,
\end{equation}
with $a_i(x)\in\Z_p[x]$, $\deg a_i(x)<m$. To any coefficient $a_i(x)$ we attach the $p$-adic value $u_i=v_p(a_i(x))\in\Z\cup\{\infty\}$, and the point of the plane $(i,u_i)$, if $u_i<\infty$.

\begin{definition}
The $\phi$-Newton polygon of $f(x)$ is the lower convex envelope of the set of points $(i,u_i)$, $u_i<\infty$, in the Euclidian plane.
We denote this open polygon by $\nph(f)$.
\end{definition}

The length of this polygon is by definition the abscissa of the last vertex. We denote it by  $\ell(\nph(f)):=r=\lfloor \deg(f)/m\rfloor$. Note that $\deg f(x)=m r+\deg a_r(x)$.

Usually, $f(x)$ will be a monic polynomial not divisible by $\phi(x)$ in $\Z_p[x]$ (i.e. $a_0(x)\ne0$); in this case, the typical shape of this polygon is shown in the following figure.

\begin{center}
\setlength{\unitlength}{5.mm}
\begin{picture}(12,8)
\put(7.85,-.15){$\bullet$}\put(5.85,.85){$\bullet$}\put(4.85,-.15){$\bullet$}\put(2.85,1.85){$\bullet$}
\put(1.85,1.85){$\bullet$}\put(.85,3.85){$\bullet$}\put(-.15,5.85){$\bullet$}
\put(-1,0){\line(1,0){12}}\put(0,-1){\line(0,1){8}}
\put(5,0){\line(-3,2){3}}\put(2,2){\line(-1,2){2}}\put(5,.03){\line(-3,2){3}}
\put(2,2.03){\line(-1,2){2}}\put(8,0){\line(-1,0){3}}\put(8,.02){\line(-1,0){3}}
\put(9.85,1.85){$\bullet$}\put(8,0){\line(1,1){2}}\put(8,.02){\line(1,1){2}}
\put(9.9,-.8){\begin{footnotesize}$r$\end{footnotesize}}
\put(4,-.8){\begin{footnotesize}$\ord_{\overline{\phi}}\left(\overline{f}\right)$\end{footnotesize}}
\put(-.4,-.6){\begin{footnotesize}$0$\end{footnotesize}}
\put(5,4.6){\begin{footnotesize}$\nph(f)$\end{footnotesize}}
\multiput(10,-.1)(0,.25){9}{\vrule height2pt}
\end{picture}\qquad\qquad
\begin{picture}(7,8)
\put(4.85,-.15){$\bullet$}\put(2.85,1.85){$\bullet$}
\put(1.85,1.85){$\bullet$}\put(.85,3.85){$\bullet$}\put(-.15,5.85){$\bullet$}
\put(-1,0){\line(1,0){8}}\put(0,-1){\line(0,1){8}}
\put(5,0){\line(-3,2){3}}\put(2,2){\line(-1,2){2}}\put(5,.03){\line(-3,2){3}}
\put(2,2.03){\line(-1,2){2}}
\put(4,-.8){\begin{footnotesize}$\ord_{\overline{\phi}}\left(\overline{f}\right)$\end{footnotesize}}
\put(-.4,-.6){\begin{footnotesize}$0$\end{footnotesize}}
\put(2.6,4.6){\begin{footnotesize}$\npp(f)$\end{footnotesize}}
\end{picture}
\end{center}\be\medskip
\begin{center}
Figure 1
\end{center}
\be

The $\phi$-Newton polygon is the union of different adjacent \emph{sides} $S_1,\dots,S_g$ with increasing slope $\la_1<\la_2<\cdots<\la_g$.
We shall write $\nph(f)=S_1+\cdots+S_g$. The end points of the sides are called the \emph{vertices} of the polygon.

\begin{definition}
The polygon determined by the sides of negative slope of $\nph(f)$ is called the \emph{principal $\phi$-polygon} of $f(x)$ and will be denoted by $\npp(f)$. The length of $\npp(f)$ is always equal to the highest exponent $a=\ord_{\overline{\phi}}\left(\overline{f}\right)$ such that $\overline{\phi(x)}{}^{\,a}$ divides $\overline{f(x)}$ in $\fp[x]$.
\end{definition}

\begin{definition}\label{phindex}
The \emph{$\phi$-index of $f(x)$} is $\deg \phi$ times the number of points with integral coordinates that lie below or on the polygon $N_{\phi}^-(f)$, strictly above the horizontal axis, and strictly beyond the vertical axis. We denote this number by $\ \ind_{\phi}(f)$.
\end{definition}

From now on, we denote $N=\npp(f)$ for simplicity. The principal polygon $N$ and the set of points $(i,u_i)$ that lie on $N$, contain the arithmetic information we are interested in. Note that, by construction, these points $(i,u_i)$ lie all above $N$.

We attach to any abscissa $0\le i\le \ell(N)$ the following \emph{residual coefficient} $c_i\in\fph$:
$$\as{1.6}
c_i=\left\{\begin{array}{ll}
0,&\mbox{ if $(i,u_i)$ lies strictly above $N$ or }u_i=\infty,\\\rd\left(\dfrac{a_i(x)}{p^{u_i}}\right),&\mbox{ if $(i,u_i)$ lies on }N.
\end{array}
\right.
$$ Note that $c_i$ is always nonzero in the latter case, because $\deg a_i(x)<m$. 

Let $S$ be one of the sides of $N$, with slope $\la$. Let $\la=-h/e$, with $h,e$ positive coprime integers. We introduce the following notations:
\begin{enumerate}
\item the \emph{length of $S$} is the length, $\ell(S)$, of the projection of $S$ to the $x$-axis,
\item the \emph{degree of $S$} is $d(S):=\ell(S)/e$,
\item the \emph{ramification index of $S$} is $e(S):=e=\ell(S)/d(S)$.
\end{enumerate} 

Note that $S$ is divided into $d(S)$ segments by the
points of integer coordinates that lie on  $S$. These points contain important arithmetic information, encoded by a polynomial that is built with the
coefficients of the $\phi$-adic development of $f(x)$ to whom these points are attached.

\begin{definition}
Let $s$ be the initial abscissa of $S$, and let $d:=d(S)$. We define the \emph{residual polynomial} attached to $S$ (or to $\la$) to be the polynomial:
$$
R_{\la}(f)(y):=c_s+c_{s+e}\,y+\cdots+c_{s+(d-1)e}\,y^{d-1}+c_{s+de}\,y^d\in\fph[y].
$$
\end{definition}

\noindent{\bf Remarks. }
\begin{enumerate}
\item Note that $c_s$ and $c_{s+de}$ are always nonzero, so that the residual polynomial has degree $d$ and it is never  divisible by $y$.
\item The residual polynomial depends on $\phi(x)$, but we omit the reference to $\phi(x)$ in the notation, since this polynomial will be specified in each context. 
\end{enumerate}

\subsection{Arithmetic applications of Newton polygons}
In this section we recall some results of \O.\ Ore on arithmetic applications of Newton polygons \cite{ore0,ore2}. Modern proofs of these results can be found in \cite[Sec.1]{GMN}.

We fix from now on a monic irreducible polynomial $f(x)\in\Z[x]$ of degree $n$, and a root  $\t\in\qb$ of $f(x)$. We denote by $K=\Q(\t)$ the number field generated by $\t$ and by $\ZK$ the ring of integers of $K$. Let
$$\ind_p(f):=v_p\left(\left(\ZK\colon\Z[\t]\right)\right),$$ be the $p$-adic value of the index of the polynomial $f(x)$. Recall the well-known relationship, $v_p(\disc(f))=2\ind_p(f)+v_p(\disc(K))$, between $\ind_p(f)$, the discriminant of $f(x)$ and the discriminant of $K$.

Let $\pp$ be the set of prime ideals of $K$ lying above $p$.
For any $\p\in\pp$, we denote by $v_{\p}$ the discrete valuation of $K$ associated to $\p$, and by $e(\p/p)$ the ramification index of $\p$. Endow $K$ with the $\p$-adic topology and fix a topological embedding $\iota_\p\colon K\hookrightarrow \qb_p$; we have then,
$$
v_\p(\alpha)=e(\p/p)v_p(\iota_\p(\alpha)),\quad\forall \,\alpha\in K.
$$
In particular, if $\t^\p:=\iota_\p(\t)\in\qb_p$, then  
\begin{equation}\label{localglobal}
v_\p(P(\t))=e(\p/p)v_p(P(\t^\p)),\quad \forall\, P(x)\in\Z[x]. 
\end{equation}

After Hensel's work, we know that  there is a canonical bijection between $\pp$ and the set of monic irreducible factors of $f(x)$ in $\Z_p[x]$.  With the above notations, the irreducible factor attached to a prime ideal $\p$
is the minimal polynomial $F_\p(x)\in\Z_p[x]$ of $\t^\p$ over $\Q_p$.

Choose monic polynomials $\phi_1(x),\dots,\phi_t(x)\in\Z[x]$ whose reduction modulo $p$ are the different irreducible factors of $\overline{f(x)}$ in $\F_p[x]$. We have then a decomposition:
$$
f(x)\equiv \phi_1(x)^{\ell_1},\dots,\phi_t(x)^{\ell_t} \md{p},
$$for certain positive exponents $\ell_1,\dots,\ell_t$. We have $n=m_1\ell_1+\cdots+m_t\ell_t$, where $m_i:=\deg\phi_i$. By Hensel's lemma, $f(x)$ decomposes in $\Z_p[x]$ as:
$$
f(x)=F_1(x)\cdots F_t(x),
$$for certain monic factors $F_i(x)\in\Z_p[x]$ such that $F_i(x)\equiv \phi_i(x)^{\ell_i} \md{p}$.

\begin{definition}For all $1\le i\le t$, define
$\pp_{\phi_i}:=\{\p\in\pp\tq v_\p(\phi_i(\t))>0\}=\{\p\in\pp\tq v_p(\phi_i(\t^\p))>0\}$.
\end{definition}

Since the polynomials $\phi_1(x),\dots,\phi_t(x)$ are pairwise coprime modulo $p$, the set $\pp$ splits as the disjoint union: 
$\ \pp=\pp_{\phi_1}\coprod\cdots\coprod \pp_{\phi_t}$.

\begin{lemma}\label{petit}Let $\phi(x):=\phi_i(x)$, for some $1\le i\le t$.
Let $P(x)\in\Z[x]$ be a polynomial of degree less than $\deg \phi$. Then, 
$$
v_\p(P(\t))=e(\p/p)\,v_p(P(x)), \quad \forall\,\p\in\pp_\phi.
$$ 
\end{lemma}

\begin{proof}
Take $\p\in\pp_\phi$. If $P(x)=0$ the statement is obvious; suppose $P(x)\ne0$ and let $P(x)=p^\nu Q(x)$, with $\nu=v_p(P(x))$. Since $\phi(x)$ is irreducible modulo $p$, $$
v_\p(Q(\t))>0\sii  \overline{\phi(x)}\mbox{ divides }\overline{Q(x)}\mbox{ in }\F_p[x].
$$ 
Since $\overline{Q(x)}\ne0$ and $\deg \overline{Q}\le\deg Q<\deg \phi=\deg\overline{\phi}$, we see that $v_\p(Q(\t))=0$, and:
$$
v_\p(P(\t))=v_\p(p^\nu Q(\t)))=v_\p(p^\nu)=e(\p/p)\,v_p(P(x)).
$$
\end{proof}

\noindent{\bf Notation. }If $\F$ is a finite field and $\varphi(y),\,\psi(y)\in\F[y]$, we write $\varphi\sim\psi$ to indicate that the two polynomials coincide up to multiplication by a nonzero constant in $\F$.

\begin{theorem}[\bf Theorem of Ore]\label{ore}
With the above notations, let $\phi(x):=\phi_i(x)$, $F(x):=F_i(x)$, for some $1\le i\le t$. Suppose that $$\npp(f)=S_1+\cdots+S_g$$ has $g$ different sides with slopes $\la_1<\cdots<\la_g$; then $F(x)$ admits a factorization in $\zpx$ into a product of $g$ monic polynomials
$$
F(x)=G_1(x)\cdots G_g(x),
$$
such that, for all $1\le j\le g$, 
\begin{enumerate}
 \item $\nph(G_j)$ is one-sided, with slope $\la_j$,
\item $R_{\la_j}(G_j)(y)\sim R_{\la_j}(F)(y)\sim R_{\la_j}(f)(y)$,
\item All roots $\alpha\in\qb_p$ of $G_j(x)$ satisfy $v_p(\phi(\alpha))=|\la_j|$. 
\end{enumerate}

Moreover, for all $j$, let $$R_{\la_j}(f)(y)=\psi_{j,1}(y)^{n_{j,1}}\cdots \psi_{j,r_j}(y)^{n_{j,r_j}}$$ be the decomposition of $R_{\la_j}(f)(y)$ into a product of powers of pairwise different irreducible monic polynomials in $\fph[y]$; then, the polynomial $G_j(x)$ experiments a further factorization in $\Z_p[x]$ into a product of $r_j$ monic polynomials
$$
G_j(x)=H_{j,1}(x)\cdots H_{j,r_j}(x),
$$   
such that, for all $1\le k\le r_j$, 
\begin{enumerate}
 \item $\nph(H_{j,k})$ is one-sided, with slope $\la_j$,
\item $R_{\la_j}(H_{j,k})(y)\sim \psi_{j,k}(y)^{n_{j,k}}$. 
\end{enumerate}

Finally, if  $n_{j,k}=1$, the polynomial $H_{j,k}(x)$ is irreducible in $\Z_p[x]$ and the ramification index and residual degree of the $p$-adic field $K_{j,k}$ generated by this polynomial are given by
$\
e(K_{j,k}/\Q_p)=e(S_j)$, $\, f(K_{j,k}/\Q_p)=\deg\phi\cdot\deg\psi_{j,k}$.\hfill{$\Box$}
\end{theorem}

\begin{definition}\label{pregular}
Let $\phi(x)\in\Z[x]$ be a monic polynomial,  irreducible modulo $p$. We say that $f(x)$ is \emph{$\phi$-regular} if for every side of $N_{\phi}^-(f)$, the residual polynomial attached to the side is separable.

Choose monic polynomials $\phi_1(x),\dots,\phi_t(x)\in\Z[x]$ whose reduction modulo $p$ are the different irreducible factors of $\overline{f(x)}$ in $\F_p[x]$.
We say that $f(x)$ is $p$-regular with respect to this choice if $f(x)$ is $\phi_i$-regular for every $1\le i\le t$.
\end{definition}

Although this concept depends on the choice of $\phi_1(x),\dots,\phi_t(x)$, in the sequel we shall just say ``$f(x)$ is $p$-regular", taking for granted that there is an implicit choice of these liftings to $\Z[x]$ of the monic irreducible polynomials of $\F_p[x]$ that divide $\overline{f(x)}$. Since the length of $N_{\phi_i}^-(f)$ is equal to  $\ord_{\overline{\phi_i}}(\overline{f})$, in practice we need only to check $\phi_i$-regularity for those $\phi_i(x)$ with $\ord_{\overline{\phi_i}}(\overline{f})>1$.

If $f(x)$ is $p$-regular, Theorem \ref{ore} provides the complete facto\-rization of $f(x)$ into a product of irreducible polynomials in $\Z_p[x]$, or equivalently, the decomposition of $p$ into a product of prime ideals of $K$. 
Moreover, in the $p$-regular case the $p$-index of $f(x)$ is also determined by the shape of the different $\phi$-Newton polygons.

\begin{theorem}[\bf Theorem of the index]\label{index}
$\ind_p(f)\ge\ind_{\phi_1}(f)+\cdots+\ind_{\phi_t}(f)$, and
equality holds if $f(x)$ is $p$-regular.\hfill{$\Box$}
\end{theorem}

\section{Computation of a $p$-integral basis in the regular case}\label{secMain}
We keep the above notations for $f(x)$, $\t$, $K$, $\pp$. Recall the decomposition:
$$
f(x)\equiv \phi_1(x)^{\ell_1},\dots,\phi_t(x)^{\ell_t} \md{p},
$$
for a certain choice of monic polynomials $\phi_i(x)\in\Z[x]$ of degree $m_i$, whose reduction modulo $p$ are the different irreducible factors of $\overline{f(x)}$ in $\F_p[x]$. We fix one of these monic polynomials $\phi(x)=\phi_i(x)$.

\begin{definition}
The \emph{quotients attached to the  $\phi$-adic development} (\ref{phiadic}) of $f(x)$ are, by definition, the different quotients $q_1(x),\dots,q_r(x)$ that are obtained along the computation of the coefficients of the development:
$$
\begin{array}{rcl}
f(x)&=&\phi(x)q_1(x)+a_0(x),\\
q_1(x)&=&\phi(x)q_2(x)+a_1(x),\\
\cdots&&\cdots\\
q_{r-1}(x)&=&\phi(x)q_r(x)+a_{r-1}(x),\\
q_r(x)&=&\phi(x)\cdot 0+a_r(x)=a_r(x).
\end{array}
$$
\end{definition}

Equivalently, $q_j(x)$ is the quotient of the division of $f(x)$ by $\phi(x)^j$; we denote by $r_j(x)$ the residue of this division. Thus, for all $1\le j\le r$ we have 
\begin{equation}\label{division}
f(x)=r_j(x)+q_j(x)\phi(x)^j.
\end{equation}
\begin{equation}\label{residue}
r_j(x)=a_0(x)+a_1(x)\phi(x)+\cdots+a_{j-1}(x)\phi(x)^{j-1}.
\end{equation}
\begin{equation}\label{quotient}
q_j(x)=a_j(x)+a_{j+1}(x)\phi(x)+\cdots+a_r(x)\phi(x)^{r-j}.
\end{equation}

Our first aim is to compute, for some quotients $q(x)$ attached to the $\phi$-adic development of $f(x)$, the highest power $p^a$ of $p$ such that $q(\t)/p^a$ is integral.

Let $N_\phi^-(f)=S_1+\cdots+S_g$ be the principal $\phi$-Newton polygon of $f(x)$, and denote $\ell=\ell(N_\phi^-(f))=\ord_{\overline{\phi}}(\overline{f})$. For any integer abscissa $0 \le j\le \ell$, let $y_j\in\Q$ be the ordinate of the point of $N_{\phi}^-(f)$ of abscissa $j$. These rational numbers form an strictly decreasing sequence, and $y_\ell=0$. Note that, by Definition \ref{phindex},
$$
\ind_{\phi}(f)=\lfloor y_1\rfloor+\cdots+\lfloor y_{\ell-1}\rfloor.
$$

\begin{lemma}\label{convex}
For any integer abscissa $0\le j\le \ell$, let  $1\le s_j\le g$ be the greatest index such that the projection of $S_{s_j}$ to the $x$-axis contains $j$. Then, 
$$
y_j\le y_k+(k-j)|\la_{s_j}|, \quad \forall\,0 \le k\le \ell.
$$  
\end{lemma}

\begin{proof}
This is an immediate consequence of the convexity of $N_\phi^-(f)$. Figure 2 illustrates the two different situations that arise according to $k\ge j$ or $k<j$.
\end{proof}
\begin{center}
\setlength{\unitlength}{5.mm}
\begin{picture}(12,8)
\put(5.85,4.85){$\bullet$}\put(7.85,1.85){$\bullet$}\put(10.85,.85){$\bullet$}
\put(3,0){\line(1,0){9}}\put(4,-1){\line(0,1){8}}
\put(6,5){\line(2,-3){2}}\put(6,5.03){\line(2,-3){2}}
\put(8,2){\line(3,-1){3}}\put(8,2.03){\line(3,-1){3}}
\multiput(7.05,5.7)(.1,-.15){30}{\mbox{\tiny .}}
\put(9.9,-.7){\begin{footnotesize}$k$\end{footnotesize}}
\put(6.9,-.7){\begin{footnotesize}$j$\end{footnotesize}}
\put(3.6,-.7){\begin{footnotesize}$0$\end{footnotesize}}
\put(3.2,3.4){\begin{footnotesize}$y_j$\end{footnotesize}}
\put(3.2,1.2){\begin{footnotesize}$y_k$\end{footnotesize}}
\put(-.5,5.7){\begin{footnotesize}$y_k+(k-j)|\la_{s_j}|$\end{footnotesize}}
\multiput(7,-.1)(0,.25){28}{\vrule height2pt}
\multiput(10.1,-.1)(0,.25){6}{\vrule height2pt}
\multiput(3.9,5.8)(.25,0){13}{\hbox to 2pt{\hrulefill }}
\multiput(3.9,1.3)(.25,0){25}{\hbox to 2pt{\hrulefill }}
\multiput(3.9,3.5)(.25,0){13}{\hbox to 2pt{\hrulefill }}
\end{picture}
\begin{picture}(12,8)
\put(5.85,4.85){$\bullet$}\put(7.85,1.85){$\bullet$}\put(10.85,.85){$\bullet$}
\put(3,0){\line(1,0){9}}\put(4,-1){\line(0,1){8}}
\put(6,5){\line(2,-3){2}}\put(6,5.03){\line(2,-3){2}}
\put(8,2){\line(3,-1){3}}\put(8,2.03){\line(3,-1){3}}
\multiput(7,3.45)(.12,-.04){26}{\mbox{\tiny .}}
\put(9.8,-.7){\begin{footnotesize}$j$\end{footnotesize}}
\put(6.9,-.7){\begin{footnotesize}$k$\end{footnotesize}}
\put(3.6,-.7){\begin{footnotesize}$0$\end{footnotesize}}
\put(3.2,3.45){\begin{footnotesize}$y_k$\end{footnotesize}}
\put(3.2,1.2){\begin{footnotesize}$y_j$\end{footnotesize}}
\put(-.5,2.4){\begin{footnotesize}$y_k+(k-j)|\la_{s_j}|$\end{footnotesize}}
\multiput(7,-.1)(0,.25){15}{\vrule height2pt}
\multiput(10,-.1)(0,.25){20}{\vrule height2pt}
\multiput(3.9,2.5)(.25,0){25}{\hbox to 2pt{\hrulefill }}
\multiput(3.9,1.3)(.25,0){25}{\hbox to 2pt{\hrulefill }}
\multiput(3.9,3.5)(.25,0){13}{\hbox to 2pt{\hrulefill }}
\end{picture}
\end{center}\medskip
\begin{center}
Figure 2
\end{center}

\begin{proposition}\label{denominator}
For all $\ 0< j<\ell$, we have $\ q_j(\t)/p^{\lfloor y_j\rfloor}\in\Z_K$.
\end{proposition}

\begin{proof}
Fix $0< j<\ell$. We shall show that $v_{\p}(q_j(\t))\ge
e(\p/p)y_j$ for all $\p\in\pp$.

Let $\la_1<\cdots<\la_g$ be the slopes of the different sides of
$N_\phi^-(f)=S_1+\cdots+S_g$. By the Theorem of Ore, the set
$\pp_{\phi}$ splits into the disjoint union of the $g$ subsets
$$\pp_{S_s}=\{\p\in\pp\tq v_{\p}(\phi(\t))=e(\p/p)|\la_s|\}, \quad 1\le s\le g.$$
 Let $1\le s_j\le g$ be the greatest
integer such that the projection of $S_{s_j}$ to the horizontal axis contains the abscissa $j$.

 Suppose first that
$\p\in\pp_{S_s}$ for some $s\le s_j$; in this case, Lemma \ref{petit},  (\ref{localglobal}) and the Theorem of Ore show that, for all $k\ge j$:
\begin{align*}
v_{\p}(a_k(\t)\phi(\t)^{k-j})=&\ e(\p/p)(u_k+(k-j)|\la_s|)\ge e(\p/p)(u_k+(k-j)|\la_{s_j}|)\\&\ \ge e(\p/p)(y_k+(k-j)|\la_{s_j}|)  \ge e(\p/p)y_j,
\end{align*}the last inequality by Lemma \ref{convex}.
Hence, after the substitution $x=\t$, each summand in (\ref{quotient}) has $v_{\p}$-value greater than or equal to $e(\p/p)y_j$, so that $v_{\p}(q_j(\t))\ge
e(\p/p)y_j$ as well.

Suppose now that either $\p\nmid\phi(\t)$ or $\p\in \pp_{S_s}$ for
some $s>s_j$; that is, $v_{\p}(\phi(\t))=e(\p/p)\mu$ for some
$\mu<|\la_{s_j}|$ ($\mu=0$ if $\p\nmid\phi(\t)$ and $\mu=|\la_s|$ if $\p\in
\pp_{S_s}$). In this case, we use the identities (\ref{division}) and (\ref{residue}),
which imply, again by Lemma \ref{petit} and the Theorem of Ore:
\begin{align*}
v_\p(q_j(\t))=&\ v_p(r_j(\t))-je(\p/p)\mu\ge\min_{0\le k<j}\{v_{\p}(a_k(\t)\phi(\t)^k)\}-je(\p/p)\mu\\ 
=&\ v_{\p}(a_{k_0}(\t)\phi(\t)^{k_0})-je(\p/p)\mu=e(\p/p)(u_{k_0}-(j-k_0)\mu)\\
>&\ e(\p/p)(u_{k_0}-(j-k_0)|\la_{s_j}|)\ge e(\p/p)(y_{k_0}-(j-k_0)|\la_{s_j}|)\ge e(\p/p)y_j,
\end{align*}the last inequality by Lemma \ref{convex}.
\end{proof}

From now on we shall denote by $q_{i,1}(x),\,\dots,\,q_{i,\ell_i}(x)\in\Z[x]$ the first $\ell_i$ quotients attached to the $\phi_i$-adic development of $f(x)$. Also, we denote by $r_{i,j}(x)\in\Z[x]$ the residue of the division of $q_{i,j}(x)$ by $\phi_i(x)^j$, as in (\ref{division}). Finally, we denote by $y_{i,1},\,\dots,\,y_{i,\ell_i}\in\Q$ the ordinates of the points lying on $N_{\phi_i}^-(f)$ with integer abscissa. 

\begin{lemma}\label{ordphi}
For all $1\le I\le t$, the quotients $q_{i,1}(x),\,\dots,\,q_{i,\ell_i}(x)$ satisfy
$$
\ord_{\overline{\phi_I}}(\overline{q_{i,j}})=\left\{\begin{array}{ll}
 \ell_i-j,&\qquad \mbox{if }I=i,\\
\ell_I,&\qquad \mbox{if }I\ne i.
\end{array}
\right.
$$
\end{lemma}

\begin{proof}
The first coefficient of the $\phi_i$-development of $f(x)$ that is not divisible by $p$ is  $a_{\ell_i}(x)$. Hence, $\ord_{\overline{\phi_i}}(\overline{q_{i,j}})=\ell_i-j$ by (\ref{quotient}), and $\overline{r_{i,j}}=0$ for all $i$ and all $1\le j\le \ell_i$ by (\ref{residue}). Therefore, if $I\ne i$, then $\ell_I=\ord_{\overline{\phi_I}}(f)=\ord_{\overline{\phi_I}}(\overline{q_{i,j}})$ by
 (\ref{division}).
\end{proof}

\begin{lemma}\label{numerators}
The following elements in $\Z[\t]$ are a $\zp$-basis of $\zp[\t]$:
$$
\alpha_{i,j,k}:=q_{i,j}(\t)\,\t^k\in\Z[\t],\quad 1\le i\le t,\ 1\le j\le\ell_i,\ 0\le k<m_i.
$$  
\end{lemma}

\begin{proof}
Let us show first that the $n=\ell_1m_1+\cdots+\ell_tm_t$ polynomials $q_{i,j}(x)x^k$ are linearly independent modulo $p$. Denote by $Q_{i,j,k}(x)=\overline{q_{i,j}(x)x^k}$ their reduction modulo $p$, and suppose that for some constants $a_{i,j,k}\in\F_p$ we have
\begin{equation}\label{li}
\sum_{i,j,k}a_{i,j,k}Q_{i,j,k}(x)=0.
\end{equation}
Consider the polynomials 
\begin{equation}\label{aij}
A_{i,j}(x):=\sum_{0\le k<m_i}a_{i,j,k}Q_{i,j,k}(x)=\overline{q_{i,j}(x)}\sum_{0\le k<m_i}a_{i,j,k}x^k\in\F_p[x].
\end{equation}
Now, the equality (\ref{li}) is equivalent to $\sum_{i,j}A_{i,j}(x)=0$, and this implies $A_{i,j}(x)=0$ for all $i,j$. In fact, if $A_{i,j}(x)\ne0$,  (\ref{ordphi}) shows that 
\begin{equation}\label{ordphi2}
\ord_{\overline{\phi_I}}(\overline{A_{i,j}})=\left\{\begin{array}{ll}
 \ell_i-j,&\qquad \mbox{if }I=i,\\
\ge \ell_I,&\qquad \mbox{if }I\ne i.
\end{array}
\right.
\end{equation}
Reorder $\phi_1,\dots,\phi_t(x)$ by decreasing length of the principal part of the Newton polygon: $\ell_1\ge\cdots\ge \ell_t$, and let $(i_0,j_0)$ be the greatest pair of indices, in the lexicographical order, with $A_{i_0,j_0}(x)\ne 0$; then, $\ord_{\overline{\phi_{i_0}}}(\sum_{i,j}A_{i,j}(x))=\ell_{i_0}-j_0$, which is a contradiction. From (\ref{aij}) we deduce $a_{i,j,k}=0$ for all $i,j,k$, so that the polynomials  
$Q_{i,j,k}(x)$ are $\F_p$-linearly independent.

In particular, since the polynomials $q_{i,j}(x)x^k$ have all degree less than $n$, the integral elements $\alpha_{i,j,k}$ are $\Z$-linearly independent. Let $M$ be the $\Z$-module genera\-ted by all $\alpha_{i,j,k}$. In order to finish the proof of the lemma we need only to show that $p$ does not divide the index $(\Z[\t]\colon M)$. 

Take $g(x)\in\Z[x]$ of degree less than $n$ and suppose that $pg(\t)\in M$. We have 
$$
pg(x)=\sum_{i,j,k}a_{i,j,k}q_{i,j}(x)x^k,
$$
for certain integers $a_{i,j,k}$. Since $Q_{i,j,k}(x)$ are $\F_p$-linearly independent, all $a_{i,j,k}$ are divisible by $p$ and we deduce that $g(\t)\in M$.
\end{proof}

\begin{theorem}\label{main}
If $f(x)$ is $p$-regular, then the family of all
$\alpha_{i,j,k}/p^{\lfloor y_{i,j}\rfloor}$ is a $p$-integral
basis of $\Z_K$.
\end{theorem}

\begin{proof}Let $M:=\gen{\alpha_{i,j,k}}_{\Z}\subseteq \Z[\t]$, and consider the chain of free $\zp$-modules
$$
\zp[\t]= M\otimes_\Z\zp\subseteq \gen{\alpha_{i,j,k}/p^{\lfloor y_{i,j}\rfloor}}_\Z\otimes_\Z\zp\subseteq \Z_K\otimes_\Z\zp
$$
The $v_p$ value of the index of $M\otimes_\Z\zp$ inside $\gen{\alpha_{i,j,k}/p^{\lfloor y_{i,j}\rfloor}}_\Z\otimes_\Z\zp$ is exactly $\ind_{\phi_1}(f)+\cdots+\ind_{\phi_t}(f)$, which is equal to $\ind_p(f)$ by the Theorem of the index. Hence, $\gen{\alpha_{i,j,k}/p^{\lfloor y_{i,j}\rfloor}}_\Z\otimes_\Z\zp= \Z_K\otimes_\Z\zp$. 
\end{proof}

This theorem covers a wide range of cases, because there is a high probability that a random monic polynomial with integer coefficients is $p$-regular with respect to randomly chosen lifts to $\Z[x]$ of the irreducible factors modulo $p$. To illustrate its usefulness, we apply the theorem to  compute a $p$-integral basis of an arbitray quartic number field. To this aim is devoted the rest of the paper.

\section{Quartic fields: preliminaries}\label{secPrel}
From now on we fix a prime number $p$ and a monic irreducible polynomial $$f(x)=x^4+ax^2+bx+c\in\Z[x].$$ We choose a root $\t\in\qb$ of $f(x)$ and we denote by $K$ the quartic field generated by $\t$. 
The discriminant of $f(x)$ is
$$
\disc(f)=16a^4c-128a^2c^2+144ab^2c-4a^3b^2+256c^3-27b^4.
$$

We want to apply the methods of the last section to compute a $p$-integral basis of $K$. Clearly, if $p$ does not divide $\disc(f)$, then $1,\t,\t^2,\t^3$ is a $p$-integral basis; thus, we are interested only in the case $p|\disc(f)$, or equivalently, $f(x)$ inseparable modulo $p$.  We shall discuss separately the different cases that arise according to the different possibilities for the factorization of $f(x)$ modulo $p$. These cases can be distinguished by the computation of $\gcd(\overline{f},\overline{f'})$; however, in some concrete applications it might be useful to distinguish them directly in terms of $a,b,c$. This is the aim of the next lemma.
We denote by $\leg{-}{p}$ the Legendre symbol.

\begin{lemma}\label{factorization}
 Let $f(x)=x^4+ax^2+bx+c\in\Z[x]$. Then,

(A)  The polynomial $f(x)$ factorizes modulo $p$ as the square of a quadratic irreducible polynomial if and only if it satisfies one of the following conditions:
\begin{enumerate}
\item[(A1)]$p>2,\ p\mid b,\ p\nmid ac, \ a^2\equiv 4c\md{p}$,\ $\leg{-a/2}p=-1$.
\item[(A2)] $p=2, \ 2\mid b, \ 2\nmid ac$, 
\end{enumerate}

(B)  The polynomial $f(x)$ has only one double root modulo $p$ if and only if $p>2$ and it satisfies one of the following conditions:
\begin{enumerate}
 \item[(B1)] $p\nmid a$, \ $p\mid b$, \ $p\mid c$, 
\item[(B2)]  $p\nmid a$, \ $p\nmid b$, \ $p\mid c$, \ $4a^3+27b^2\equiv0\md{p}$,
\item[(B3)]  $p\mid a$, \ $p\nmid b$, \ $p\nmid c$, \ $256c^3\equiv 27b^4\md{p}$,
\item[(B4)]  $p\nmid abc$, \ $p\mid \disc(f)$, \ $p\nmid 2a(a^2-4c)+9b^2$.
\end{enumerate}

(C) The polynomial $f(x)$ has two different double roots modulo $p$ if and only if it satisfies one of the following conditions:
\begin{enumerate}
\item[(C1)]  $p>2$, \ $p\nmid a$, \ $p\mid b$, \ $p\nmid c$, \ $a^2\equiv 4c\md{p}$, \ $\leg{-a/2}p=1$.
\item[(C2)] $p=2$, \ $2\nmid a$, $2\mid b$, $2\mid c$, 
\end{enumerate}

(D) The polynomial $f(x)$ has a triple (and not $4$-tuple) root modulo $p$ if and only if it satisfies one of the following conditions:
\begin{enumerate}
\item[(D1)] $p>3$, \ $p\nmid abc$, \ $p\mid\disc(f)$, \ $p\mid  2a(a^2-4c)+9b^2$.
\item[(D2)] $p=3$, \ $3\mid a$, \ $3\nmid b$, \ $3\mid c$, 
\end{enumerate}

(E) The polynomial $f(x)$ has a $4$-tuple root modulo $p$ if and only if it satisfies one of the following conditions:
\begin{enumerate}
 \item[(E1)] $p\mid a$, \ $p\mid b$, \ $p\mid c$, 
\item[(E2)] $p=2$, \ $2\mid a$, \ $2\mid b$, \ $2\nmid c$.
\end{enumerate}
\end{lemma}

\begin{proof}
(A) If $p=2$, the only possibility is $f(x)\equiv(x^2+x+1)^2\equiv x^4+x^2+1\md2$. If $p>2$, this case is determined by the existence of an integer $s$ such that $x^2+s$ is irreducible modulo $p$ and $f(x)\equiv(x^2+s)^2\md{p}$, or equivalently
$$
a\equiv 2s\md{p},\quad b\equiv 0\md{p},\quad c\equiv s^2\md{p},\quad 
\leg{-s}p=-1,
$$which is is equivalent to (A1).

(B) This case does not occur for $p=2$. For $p>2$, it is determined by the existence of two integers $s,t\in\Z$ such that 
\begin{equation}\label{onedouble}
f(x)\equiv (x-s)^2(x^2+2sx+t)\md{p}, \quad t\not\equiv-3s^2\md{p}, \quad t\not\equiv s^2\md{p}.
\end{equation}
The decomposition of $f(x)$ modulo $p$ implies that
\begin{equation}\label{onedoubleabc}
a\equiv t-3s^2\md{p},\quad b\equiv 2s(s^2-t)\md{p},\quad c\equiv s^2t\md{p}. 
\end{equation}

Suppose first that $p\mid c$. If $p\mid b$, the double root of $f(x)$ modulo $p$ is zero,
and this corresponds to (B1). If $p\nmid b$, then (\ref{onedouble}) is equivalent to the fact that  $x^3+ax+b$ has a double root 
modulo $p$ and no triple root, and this is equivalent to (B2).
 
Suppose now that $p\nmid c$. If $p\mid b$, then (\ref{onedoubleabc}) shows that  $p\mid (s^2-t)$, in contradiction with (\ref{onedouble}). If $p\nmid b$ and $p\mid \disc(f)$, then $f(x)$ has some multiple root $s$
 modulo $p$ because it does not fall in case (A); hence, $a,b,c$ satisfy
(\ref{onedoubleabc}) for some integer $t$.  If $p\mid a$, the condition 
$p\mid \disc(f)$ is equivalent to $256c^3\equiv 27b^4\md{p}$ and $f(x)$ cannot have neither triple nor 
$4$-tuple roots modulo $p$ because $f'(x)\equiv 4x^3+b\md{p}$ is separable modulo $p$.
Finally, if $p\nmid a$,  from (\ref{onedoubleabc}) we get
$2a(a^2-4c)+9b^2\equiv2(t-s^2)(t+3s^2)^2\md{p}$,
so that (\ref{onedouble}) is equivalent to $p\nmid 2a(a^2-4c)+9b^2$.

(C) If $p=2$, the only possibility is $f(x)\equiv x^2(x+1)^2\md2$, which is equivalent to (C2). 
If $p>2$,  then $f(x)$ has two double roots modulo $p$ if and only if  $f(x)\equiv (x-s)^2(x+s)^2\md{p}$, for some integer $s$ not divisible by $p$. This implies (C1), because $a\equiv-2s^2\md{p}$ and $\disc(f)\equiv 16c(a^2-4c)^2\md{p}$. Conversely, (C1) implies that $f'(x)\equiv 2x(2x^2+a)\md{p}$ is a separable polynomial modulo $p$, so that $f(x)$ has neither triple nor $4$-tuple roots modulo $p$. Since $f(x)$ falls neither in case (A) nor (B), it must have two different double roots modulo $p$.

(D) 
This case is determined by the existence of an integer $s$ such that
\begin{equation}\label{triple}
f(x)\equiv (x-s)^3(x+3s)\md{p},\quad 4s\not\equiv 0\md{p}.
\end{equation}
In particular, $p$ is necessarily odd. If $p=3$, this condition is equivalent to (D2), and then $s\equiv-b \md3$. If $p>3$, condition (\ref{triple})
is equivalent to
\begin{equation}\label{tripleabc}
a\equiv-6s^2\md{p},\quad b\equiv 8s^3\md{p},\quad c\equiv -3s^4\md{p}, 
\end{equation}for some $s$ not divisible by $p$; clearly, this implies
 $p\nmid abc$, $p\mid \disc(f)$. The condition 
$p\mid  2a(a^2-4c)+9b^2$ is necessary too, because otherwise $f(x)$ would satisfy (B4). Conversely, (D1) implies that $f(x)$ has a multiple root modulo $p$ with multiplicity three, because $f(x)$ will be inseparable modulo $p$ and it does not fall in neither of the cases (A), (B), (C) and (E) below.

(E) The only possibilities for $f(x)$ modulo $p$ are: $\ f(x)\equiv (x+1)^4\md2$, and $f(x)\equiv x^4\md{p}$.
\end{proof}

A user's guide to find the $p$-integral basis of $K$ would be: check first which of the conditions (A), (B), (C), (D) or (E) of Lemma \ref{factorization} satisfies $f(x)$, and then look for the following equation, lemma or table:\medskip

\begin{center}
\begin{tabular}{|c|c|c|c|c|c|c|c|c|c|}
\hline
 A1&A2&B&C1&C2&D1&D2&E1&E2\\\hline
(\ref{Tquadraticp})&(\ref{Tquadratic2})&(\ref{Tdouble})&(\ref{Tdoubledoublep})&(\ref{Tdoubledouble2})&Lem. \ref{Ttriple}&Lem. \ref{Ttriple3}, \ref{Ttriple3bis}&Table \ref{T4tuple}&Table \ref{T4extra}\\\hline
\end{tabular}
\end{center}

\subsection{An iteration method}
Consider now an arbitrary monic and irreducible polynomial $F(x)\in\Z[x]$ of degree four.
We describe in this paragraph an iterative process that converges in some cases to an integer $s\in\Z$ such that $F(x)$ is $(x-s)$-regular.

\begin{definition}\label{regular}
We say that the integer $s$ is  \emph{regular} if  $F(x)$ is $(x-s)$-regular. Otherwise, we say that $s$ is \emph{irregular}. 
\end{definition}

Let $s$ be an integer, and denote
\begin{equation}\label{rescoeffs}\as{1.4}
\begin{array}{l}
u_0:=v_p(F(s)),\ u_1:=v_p(F'(s)),\ u_2:=v_p(F''(s)/2),\ u_3:=v_p(F'''(s)/6),\\
\sigma_0:=F(s)/p^{u_0},\ \sigma_1:=F'(s)/p^{u_1},\ \sigma_2:=F''(s)/2p^{u_2},\ \sigma_3:=F'''(s)/6p^{u_3}. 
\end{array}
\end{equation}
The $(x-s)$-adic development of $F(x)$ is 
$$
F(x)=F(s)+F'(s)(x-s)+\frac12F''(s)(x-s)^2+\frac16F'''(s)(x-s)^3+(x-s)^4,
$$
and the $(x-s)$-Newton polygon of $F(x)$ is the lower convex envelope of the points: $(0,u_0)$, $(1,u_1)$, $(2,u_2)$, $(3,u_3)$, and $(4,0)$. 

Suppose that $s$ satisfies one of the following \emph{initial conditions}:
\begin{equation}\label{initial}\as{1.2}
\begin{array}{cl}
\mbox{(i)}& u_2=0, \ u_1>0, \ u_0>0,\\
\mbox{(ii)}&u_3=0, \ u_2>0, \ u_1>0, \ u_0>0,\\
\mbox{(iii)}&\begin{cases}u_0>2u_2, \ u_1>\frac32u_2, \ u_3 \ge\frac12 u_2>0, \ \mbox{ and }\\
\mbox{the side of $N^-_{x-s}(F)$ with end points $(2,u_2)$ and $(4,0)$}\\
\mbox{has a separable residual polynomial}      \end{cases}
\end{array}
\end{equation}

If $s$ is irregular, then the unique side whose residual polynomial is inseparable has an integer slope, and the unique multiple irreducible factor of the residual polynomial has degree one. In fact, Table \ref{irregular} displays all possible situations where an irregular $s$ satisfies these initial conditions, and Figure 3 shows the possible shapes of $N^-_{x-s}(F)$ in all these cases. When we draw two points with the same abscissa, it means that both possibilities may occur.

\begin{center}
\setlength{\unitlength}{5.mm}
\begin{picture}(16,5)
\put(-.15,1.85){$\bullet$}
\put(1.85,-.15){$\bullet$}\put(.85,2.85){$\bullet$}\put(.85,.85){$\bullet$}
\put(-1,0){\line(1,0){5.5}}
\put(0,-.6){\line(0,1){4.5}}
\put(2,0){\line(-1,1){2}}\put(2.02,0){\line(-1,1){2}}
\put(-.4,-.6){\begin{footnotesize}$0$\end{footnotesize}}
\put(.9,-.6){\begin{footnotesize}$1$\end{footnotesize}}
\put(1.9,-.6){\begin{footnotesize}$2$\end{footnotesize}}
\put(2.9,-.6){\begin{footnotesize}$3$\end{footnotesize}}
\put(3.9,-.6){\begin{footnotesize}$4$\end{footnotesize}}
\put(1,-.1){\line(0,1){.2}}\put(3,-.1){\line(0,1){.2}}
\put(4,-.1){\line(0,1){.2}}
\put(3,3.5){\begin{footnotesize}(i)\end{footnotesize}}

\put(10.85,2.35){$\bullet$}
\put(11.85,1.35){$\bullet$}\put(11.85,2.85){$\bullet$}
\put(12.85,.35){$\bullet$}\put(13.85,1.35){$\bullet$}
\put(14.85,-.15){$\bullet$}\put(13.85,.1){$\bullet$}
\put(10,0){\line(1,0){5.5}}
\put(11,-.6){\line(0,1){4.5}}
\put(15,0){\line(-4,1){2}}\put(15.02,0){\line(-4,1){2}}
\put(13,.5){\line(-1,1){2}}\put(13.02,.5){\line(-1,1){2}}
\put(10.6,-.6){\begin{footnotesize}$0$\end{footnotesize}}
\put(11.9,-.6){\begin{footnotesize}$1$\end{footnotesize}}
\put(12.9,-.6){\begin{footnotesize}$2$\end{footnotesize}}
\put(13.9,-.6){\begin{footnotesize}$3$\end{footnotesize}}
\put(14.9,-.6){\begin{footnotesize}$4$\end{footnotesize}}
\put(14,-.1){\line(0,1){.2}}\put(12,-.1){\line(0,1){.2}}
\put(13,-.1){\line(0,1){.2}}
\put(14,3.5){\begin{footnotesize}(iii)\end{footnotesize}}
\end{picture}
\end{center}\be

\begin{center}
\setlength{\unitlength}{5.mm}
\begin{picture}(21,5)
\put(-.15,3.25){$\bullet$}\put(2.85,-.15){$\bullet$}
\put(1.85,.35){$\bullet$}\put(.85,3.25){$\bullet$}\put(.85,1.85){$\bullet$}
\put(-1,0){\line(1,0){5.5}}
\put(0,-.6){\line(0,1){4.5}}
\put(2,.5){\line(-2,3){2}}\put(2.02,.5){\line(-2,3){2}}
\put(3,0){\line(-2,1){1}}\put(3.02,0){\line(-2,1){1}}
\put(-.4,-.6){\begin{footnotesize}$0$\end{footnotesize}}
\put(.9,-.6){\begin{footnotesize}$1$\end{footnotesize}}
\put(1.9,-.6){\begin{footnotesize}$2$\end{footnotesize}}
\put(2.9,-.6){\begin{footnotesize}$3$\end{footnotesize}}
\put(3.9,-.6){\begin{footnotesize}$4$\end{footnotesize}}
\put(1,-.1){\line(0,1){.2}}\put(2,-.1){\line(0,1){.2}}
\put(4,-.1){\line(0,1){.2}}
\put(3,3.5){\begin{footnotesize}(ii)\end{footnotesize}}

\put(7.85,2.85){$\bullet$}\put(8.85,.85){$\bullet$}
\put(9.85,.35){$\bullet$}\put(9.85,1.85){$\bullet$}
\put(10.85,-.15){$\bullet$}
\put(7,0){\line(1,0){5.5}}
\put(8,-.6){\line(0,1){5}}
\put(11,0){\line(-2,1){2}}\put(11.02,0){\line(-2,1){2}}
\put(9,1){\line(-1,2){1}}\put(9.02,1){\line(-1,2){1}}
\put(7.6,-.6){\begin{footnotesize}$0$\end{footnotesize}}
\put(8.9,-.6){\begin{footnotesize}$1$\end{footnotesize}}
\put(9.9,-.6){\begin{footnotesize}$2$\end{footnotesize}}
\put(10.9,-.6){\begin{footnotesize}$3$\end{footnotesize}}
\put(11.9,-.6){\begin{footnotesize}$4$\end{footnotesize}}
\put(9,-.1){\line(0,1){.2}}\put(10,-.1){\line(0,1){.2}}
\put(12,-.1){\line(0,1){.2}}
\put(11,3.5){\begin{footnotesize}(ii)\end{footnotesize}}

\put(15.85,2.85){$\bullet$}
\put(16.85,1.85){$\bullet$}\put(16.85,2.85){$\bullet$}
\put(17.85,.85){$\bullet$}\put(17.85,1.85){$\bullet$}
\put(18.85,-.15){$\bullet$}
\put(15,0){\line(1,0){5.5}}
\put(16,-.6){\line(0,1){5}}
\put(19,0){\line(-1,1){3}}\put(19.02,0){\line(-1,1){3}}
\put(15.6,-.6){\begin{footnotesize}$0$\end{footnotesize}}
\put(16.9,-.6){\begin{footnotesize}$1$\end{footnotesize}}
\put(17.9,-.6){\begin{footnotesize}$2$\end{footnotesize}}
\put(18.9,-.6){\begin{footnotesize}$3$\end{footnotesize}}
\put(19.9,-.6){\begin{footnotesize}$4$\end{footnotesize}}
\put(20,-.1){\line(0,1){.2}}\put(17,-.1){\line(0,1){.2}}
\put(18,-.1){\line(0,1){.2}}
\put(19,3.5){\begin{footnotesize}(ii)\end{footnotesize}}
\end{picture}
\end{center}\be
\begin{center}
Figure 3
\end{center}

\begin{center}
\begin{table}
\caption{\footnotesize Irregular cases for an integer $s$ satisfying one of the initial conditions (i), (ii), or (iii) of (\ref{initial}).  }\label{irregular}
\begin{small}
\begin{tabular}{|c|c|c|c|c|c|}\hline
&$p$&$u_i,\,\sigma_i$&slope&residual polynomial\\\hline
(i)&$2$&$u_0$ even, $u_0<2u_1$&$u_0/2$&$(y+1)^2$\\\hline
(i)&$>2$&$u_0=2u_1,\,\overline{\sigma}_1^2=4\overline{\sigma}_0\overline{\sigma}_2$&$u_0/2$&$\overline{\sigma}_2\left(y+\dfrac{\overline{\sigma}_1}{2\overline{\sigma}_2}\right)^2$\\\hline
(ii)&$2$&\as{1}$\begin{array}{c}
 u_0>3u_2,\ u_0+u_2 \mbox{ even}\\u_0+u_2<2u_1\end{array}$
&$(u_0-u_2)/2$&$(y+1)^2$\\\hline
(ii)&$>2$&\as{1}$\begin{array}{c}
 u_0>3u_2,\ u_0+u_2=2u_1\\\overline{\sigma}_1^2=4\overline{\sigma}_0\overline{\sigma}_2
\end{array}$
&$(u_0-u_2)/2$&$\overline{\sigma}_2\left(y+\dfrac{\overline{\sigma}_1}{2\overline{\sigma}_2}\right)^2$\\\hline
(ii)&$2$&$u_0>(3/2)u_1,\ u_1 \mbox{ even},\ u_1<2u_2$
&$u_1/2$&$(y+1)^2$\\\hline
(ii)&$>2$&\as{1}$\begin{array}{c}
 u_0>3u_2,\ u_1=2u_2\\\overline{\sigma}_2^2=4\overline{\sigma}_1\overline{\sigma}_3
\end{array}$
&$u_1/2$&$\overline{\sigma}_3\left(y+\dfrac{\overline{\sigma}_2}{2\overline{\sigma}_3}\right)^2$\\\hline
(ii)&$2$&$
 u_0=3u_2,\ u_1=2u_2$
&$u_0/3$&$(y+1)^3$\\\hline
(ii)&$3$&$
 3|u_0,\,u_0<3u_2,\, u_0<(3/2)u_1$
&$u_0/3$&$\overline{\sigma}_3\left(y+\dfrac{\overline{\sigma}_0}{\overline{\sigma}_3}\right)^3$\\\hline
(ii)&$>3$&\as{1}$\begin{array}{c}
 u_0=3u_2,\ u_1=2u_2\\ \overline{\sigma}_1=\overline{\sigma}_2^2/3\overline{\sigma}_3, \,
 \overline{\sigma}_0=\overline{\sigma}_2^3/27\overline{\sigma}_3^2\end{array}$
&$u_0/3$&$\overline{\sigma}_3\left(y+\dfrac{\overline{\sigma}_2}{3\overline{\sigma}_3}\right)^3$\\\hline
(ii)&$>3$&\as{1}$\begin{array}{c}
u_0=(3/2)u_1<3u_2\\ 4\overline{\sigma}_1^3+27\overline{\sigma}_0^2\overline{\sigma}_3=0\end{array}$
&$u_0/3$&$\overline{\sigma}_3y^3+\overline{\sigma}_1y+\overline{\sigma}_0$\\\hline
(ii)&$>3$&\as{1}$\begin{array}{c}
u_0=3u_2<(3/2)u_1\\ 4\overline{\sigma}_2^3+27\overline{\sigma}_0\overline{\sigma}_3^2=0\end{array}$
&$u_0/3$&$\overline{\sigma}_3y^3+\overline{\sigma}_2y^2+\overline{\sigma}_0$\\\hline
(iii)&$2$&$u_0+u_2 \mbox{ even},\,u_0+u_2<2u_1$
&$(u_0-u_2)/2$&$(y+1)^2$\\\hline
(iii)&$>2$&$
u_0+u_2=2u_1,\,\overline{\sigma}_1^2=4\overline{\sigma}_0\overline{\sigma}_2$
&$(u_0-u_2)/2$&$\overline{\sigma}_2\left(y+\dfrac{\overline{\sigma}_1}{2\overline{\sigma}_2}\right)^2$\\\hline\end{tabular}
\end{small}
\end{table}
\end{center}

Any $s$ satisfying (\ref{initial})  will be regular with high probability. But, what can we do if we pick an irregular $s$? Let us show an efficient way to find a regular $s$.

\begin{lemma}\label{iteration}
Let $s$ be an integer satisfying one of the initial conditions of (\ref{initial}), and suppose that $s$ is irregular. Let $-\delta$ be the slope of unique side whose residual polynomial is inseparable. Let $y\in\Z$ be any integer whose reduction modulo $p$ is the unique multiple root of the residual polynomial, and consider the integer $\ t=s+yp^{\delta}$. Then, after a finite number of iterations $s_0:=s,s_1:=t,\dots,s_n$ of this procedure, we get a regular $s_n$.
\end{lemma}

\begin{proof}
Let us deal first with the cases where $(2,u_2)$ is a vertex of $N^-_{x-s}(F)$; that is, when $s$ satisfies (i), or (iii), or the subcase $u_0>3u_2$, $u_1>2u_2$ of (ii). In these cases, if $s$ is irregular then $\delta=(u_0-u_2)/2$ is an integer and
$$\as{1.2}
\begin{array}{l}
u_1>(u_0+u_2)/2 \ \mbox{ and }\ y  \mbox{ is odd, \ if }p=2, \\ u_1=(u_0+u_2)/2 \ \mbox{ and }\ 2\sigma_2y+\sigma_1\equiv 0\md{p}, \ \mbox{ if }p>2.
\end{array}
$$  

Taylor expansion shows that
\begin{equation}\label{taylor} 
\as{1.2}
\begin{array}{ll}
\;F(t)\,\equiv F(s)+F'(s)\,y\,p^{\delta}+\frac12F''(s)y^2p^{2\delta}&\md{p^{3\delta}},\\
\,F'(t)\equiv F'(s)+F''(s)\,y\,p^{\delta}&\md{p^{2\delta}},\\
F''(t)\equiv F''(s)+F'''(s)\,y\,p^{\delta}&\md{p^{2\delta}}.
\end{array}
\end{equation}
We get in all cases (note that $v_2(F''(s))=u_2+1$, $v_2(F'''(s))=u_3+1$, if $p=2$)
\begin{equation}\label{grow}
v_p(F(t))>u_0, \quad v_p(F'(t))>\frac{u_0+u_2}2,\quad  v_p(F''(t)/2)=u_2.
\end{equation}

In particular, $t$ satisfies again the initial conditions (\ref{initial}).
Now, if $v_p(F(t))=v_p(F(s))+1$ then $F(x)$ is already $(x-t)$-regular, because $N^-_{x-t}(F)\cap ([0,2]\times\R)$ has one side of length two but degree one (because
$v_p(F(t))+v_p(F''(t)/2)$ is odd). On the other hand, if  
$v_p(F(t))\ge v_p(F(s))+2$,  there is at least one more point of integral coordinates lying below or on $N^-_{x-t}(F)$, and 
$$\ind_{x-s}(F)<\ind_{x-t}(F)\le\ind_p(F),
$$the last inequality by the Theorem of the index (Theorem \ref{index}).
Hence, we cannot have an indefinite sequence $s_0:=s, s_1:=t,\dots, s_n,\dots$  of irregular integers, because  $\ind_{x-s_n}(F)$ grows strictly in each iteration and it is bounded by $\ind_p(F)$. 

Suppose now that $s$ satisfies (ii) and the irregularity is caused by a side with end points $(1,u_1)$ and $(3,0)$ (rows $5$ and $6$ of Table \ref{irregular}). From Taylor expansion we deduce now
$$
v_p(F(t))>\frac32u_1, \quad v_p(F'(t))>u_1,\quad  v_p(F''(t)/2)>\frac12u_1,\quad v_p(F'''(t)/6)=0.
$$
In spite of the fact that $v_p(F(t))$ can be smaller than $v_p(F(s))$, we can argue as above and deduce that either $t$ is regular, or $\ind_{x-t}(F)$ is strictly greater than $\ind_{x-s}(F)$.

Finally, suppose that $s$ satisfies (ii) and $N_{x-s}^-(F)$ has a side of length three
(rows  from $7$ to $11$ of Table \ref{irregular}). From Taylor expansion we deduce now
\begin{equation}\label{iterationthree}
v_p(F(t))>u_0, \ v_p(F'(t))>\frac23\, u_0,\  v_p\left(\dfrac{F''(t)}2\right)\ge\frac13\, u_0,\ v_p\left(\dfrac{F'''(t)}6\right)=0,
\end{equation}
and $v_p(F''(t)/2)=u_0/3$
if and only if the residual polynomial has a double (and not triple) root. We finish the argument as in the previous cases. 
\end{proof}

\begin{rem}\label{remark}
\mbox{\null}
(1) In the cases described in rows $2$, $4$ and $13$ of Table \ref{irregular}, if we choose $y$ such that $2\sigma_2y+\sigma_1\equiv 0\md{p^{\delta}}$, we get $v_p(F'(t)))\ge 2v_p(F'(s))$ and the iteration process is accelerated. \medskip

(2) In the cases described in rows $10$ and $11$ of Table \ref{irregular}, all elements $s_n$ computed along the iteration method satisfy: $v_p(F''(s_n)/2)=u_0/3$.
\end{rem}

\section{Explicit $p$-integral basis of quartic fields}\label{secComputation}
In this section we apply Theorem \ref{main} to exhibit a $p$-integral basis of the quartic number field $K$ generated by the irreducible polynomial $f(x)=x^4+ax^2+bx+c\in\Z[x]$. We shall denote throughout  $\Delta:=v_p(\disc(f))$.

We discuss separately the different cases that arise according to the type of decomposition of $f(x)$ modulo $p$. 
We apply the iteration method of Lemma \ref{iteration} only when $\Delta$ can be arbitrarily large, and even in these cases we try to give a more intrinsic description of the regular elements.

\subsection{$f(x)$ is inseparable modulo $p$, but it has no multiple roots in $\F_p$}\label{Secquadratic}Suppose $f(x)$ satisfies (A) of Lemma \ref{factorization}.
If $p>2$, we take any integer $s$ satisfying $v_p(a-2s)>\frac12v_p(4c-a^2)$. For instance, if $a$ is even we may take $s=a/2$. Now we take $\phi(x)=x^2+s$ as the lift to $\Z[x]$ of the irreducible factor of $f(x)$ modulo $p$. The $\phi$-adic development of $f(x)$ is
$$
f(x)=\phi(x)^2+(a-2s)\phi(x)+bx+((a-2s)^2+4c-a^2)/4,
$$
with attached quotients $q_1(x)=\phi(x)+a-2s$, $q_2(x)=1$. Since
$2v_p(a-2s)>v_p(4c-a^2)$, we have
$$2v_p(a-2s)>\min\{v_p(b),v_p(4c-a^2)\}=v_p(bx+((a-2s)^2+4c-a^2)/4).
$$
Hence, if we denote $\nu:=\frac12\min\{v_p(b),v_p(4c-a^2)\}\in\frac12\Z$, the principal Newton polygon $N_\phi(f)=N_\phi^-(f)$ is one-sided with slope $-\nu$ (see Figure 4).

\begin{center}
\setlength{\unitlength}{5.mm}
\begin{picture}(5,5)
\put(1.85,-.15){$\bullet$}\put(-.15,2.85){$\bullet$}
\put(.85,2.35){$\bullet$}\put(1,2.5){\vector(0,1){1}}
\put(-1,0){\line(1,0){5}}\put(0,-1){\line(0,1){5.4}}
\put(2,0){\line(-2,3){2}}\put(2,.03){\line(-2,3){2}}
\put(2,-.8){\begin{footnotesize}$2$\end{footnotesize}}
\put(.85,-.8){\begin{footnotesize}$1$\end{footnotesize}}
\put(-.4,-.6){\begin{footnotesize}$0$\end{footnotesize}}
\put(-1,2.9){\begin{footnotesize}$2\nu$\end{footnotesize}}
\multiput(1,-.1)(0,.25){11}{\vrule height2pt}
\end{picture}
\end{center}\be
\begin{center}
Figure 4
\end{center}

Take $\epsilon\in \F_{p^2}$ satisfying $\epsilon^2+\overline{s}=0$, and denote
$b_0:=b/p^{2\nu}$, $c_0:=(4c-a^2)/p^{2\nu}$. If $2\nu$ is odd, the side has degree one, whereas for $\nu\in\Z$, the residual polynomial attached to the side is $y^2+\overline{b}_0\epsilon+\frac14\overline{c}_0\in\F_{p^2}[y]$, 
which is always separable. Therefore, $f(x)$ is $p$-regular and Theorem \ref{main} shows that $1,\ \t,\ q_1(\t)/p^{\lfloor\nu\rfloor},\ \t q_1(\t)/p^{\lfloor\nu\rfloor}$ is a $p$-integral basis of $K$.
Since $(a-2s)/p^{\lfloor\nu\rfloor}$ is an integer, the following family is a $p$-integral basis too:
\begin{equation}\label{Tquadraticp}
1,\ \t,\ \dfrac{\t^2+s}{p^{\lfloor\nu\rfloor}},\ \dfrac{\t^3+s\t}{p^{\lfloor\nu\rfloor}},\qquad \nu=\frac12\min\{v_p(b),v_p(4c-a^2)\}.
\end{equation}

If $p=2$ we may take $\phi(x)=x^2+x+1\in\Z[x]$ as a lift of the irreducible factor of $f(x)$ modulo $2$. The $\phi$-adic development of $f(x)$ is
$$
f(x)=\phi(x)^2-(2x-a+1)\phi(x)+(b+1-a)x+c-a,
$$and the two quotients attached to this development are $q_1(x)=\phi(x)-(2x-a+1)=x^2-x+a$, $q_2(x)=1$. Since $v_2(2x-a+1)=1$ regardless of the (odd) value of $a$, the $\phi$-Newton polygon has three diferent  possibilities according to $v_2((b+1-a)x+c-a)=1,2$ or $\ge 3$, reflected in Figure 5.

\begin{center}
\setlength{\unitlength}{5.mm}
\begin{picture}(16,5)
\put(-.15,1.85){$\bullet$}\put(1.85,.85){$\bullet$}\put(.85,1.85){$\bullet$}
\put(-1,1){\line(1,0){4}}
\put(0,0){\line(0,1){5}}
\put(2,1){\line(-2,1){2}}\put(2,1.02){\line(-2,1){2}}
\put(-.4,.4){\begin{footnotesize}$0$\end{footnotesize}}
\put(.9,.4){\begin{footnotesize}$1$\end{footnotesize}}
\put(1.9,.4){\begin{footnotesize}$2$\end{footnotesize}}
\put(-.5,1.85){\begin{footnotesize}$1$\end{footnotesize}}
\put(-.5,2.85){\begin{footnotesize}$2$\end{footnotesize}}
\put(-.1,3){\line(1,0){.2}}\put(-.1,2){\line(1,0){.2}}
\put(1,.9){\line(0,1){.2}}

\put(6.85,2.85){$\bullet$}\put(7.85,1.85){$\bullet$}\put(8.85,.85){$\bullet$}
\put(7,0){\line(0,1){5}}
\put(6,1){\line(1,0){4}}
\put(9,1){\line(-1,1){2}}\put(9,1.02){\line(-1,1){2}}
\put(6.6,.4){\begin{footnotesize}$0$\end{footnotesize}}
\put(7.85,.4){\begin{footnotesize}$1$\end{footnotesize}}
\put(8.9,.4){\begin{footnotesize}$2$\end{footnotesize}}
\put(6.5,1.85){\begin{footnotesize}$1$\end{footnotesize}}
\put(6.5,2.85){\begin{footnotesize}$2$\end{footnotesize}}
\put(8,.9){\line(0,1){.2}}
\put(6.9,2){\line(1,0){.2}}\put(6.9,3){\line(1,0){.2}}

\put(13.85,3.85){$\bullet$}\put(14.85,1.85){$\bullet$}\put(15.85,.85){$\bullet$}
\put(14,0){\line(0,1){5}}
\put(13,1){\line(1,0){4}}
\put(16,1){\line(-1,1){1}}\put(16,1.02){\line(-1,1){1}}
\put(15,2){\line(-1,2){1}}\put(15,2.02){\line(-1,2){1}}
\put(13.6,.4){\begin{footnotesize}$0$\end{footnotesize}}
\put(14.9,.4){\begin{footnotesize}$1$\end{footnotesize}}
\put(15.9,.4){\begin{footnotesize}$2$\end{footnotesize}}
\put(13.5,1.85){\begin{footnotesize}$1$\end{footnotesize}}
\put(13.5,2.85){\begin{footnotesize}$2$\end{footnotesize}}
\put(12.8,3.85){\begin{footnotesize}$\ge3$\end{footnotesize}}
\put(14,4){\vector(0,1){.7}}\put(15,.9){\line(0,1){.2}}
\put(13.9,2){\line(1,0){.2}}\put(13.9,3){\line(1,0){.2}}\put(13.9,4){\line(1,0){.2}}
\end{picture}
\end{center}
\begin{center}
Figure 5
\end{center}

In the first and third cases, the residual polynomials attached to the sides have degree one, whereas in the second case the residual polynomial attached to the unique side of slope $-1$ is $y^2+y+1\in\F_2[y]$, which is irreducible. Therefore, $f(x)$ is always $2$-regular and Theorem \ref{main} shows that the $2$-integral basis of $K$ is: 
\begin{equation}\label{Tquadratic2}
1,\,\t,\,\t ^2,\,\t^3,\quad \mbox{ or }\quad 1,\,\t,\,\dfrac{\t^2+\t+1}2,\, \dfrac{\t^3+\t^2+\t}2,
\end{equation}
according to $\min\{v_2(b+1-a),v_2(c-a)\}=1$, or greater than one.

\subsection{$f(x)$ has only one double root modulo $p$}\label{double}
Suppose $f(x)$ satisfies (B) of Lemma \ref{factorization}, and let $s_0$ be any integer such that $f(x)\equiv (x-s_0)^2g(x)\md{p}$, for some separable (modulo $p$) polynomial $g(x)$ such that $g(s_0)$ is not divisible by $p$. By the Theorem of Ore (Theorem \ref{ore}), if $p$ ramifies in $\Z_K$ then it splits as $p\Z_K=\p^ 2\q_1\q_2$ (if $g(x)$ splits modulo $p$), or  $p\Z_K=\p^ 2\q$ (if $g(x)$ is irreducible modulo $p$). In any case, since $p>2$ we have $v_p(\disc(K))=0,1$ and $\ind_p(f)=\lfloor \Delta/2\rfloor$.

Since $s_0$ is a separable root of $f'(x)$ modulo $p$, Hensel's lemma shows that $f'(x)$ has a $p$-adic root congruent to $s_0$ modulo $p$. Let $s\in\Z$ be a sufficiently good approximation of this root; more precisely: $v_p(f'(s))>\Delta/2$.
The polygon $N=N_{x-s}^-(f)$ has length two and
$$
v_p(f'(s))>\Delta/2\ge \ind_p(f)\ge \ind_{x-s}(f)=\min\{\frac12v_p(f(s)),v_p(f'(s))\}.
$$
Hence, $v_p(f'(s))>v_p(f(s))/2$, and this implies that $f(x)$ is $(x-s)$-regular. In particular, $\ind_{x-s}(f)=\ind_p(f)=\lfloor \Delta/2\rfloor$.

The first quotient of the $(x-s)$-adic development is:
\begin{equation}\label{q1}
q_{1,x-s}(x)=(f(x)-f(s))/(x-s)= x^3+sx^2+(a+s^2)x+s^3+as+b.
\end{equation}
Let $q_2(x)$ be the second quotient of this $(x-s)$-development, and let $q_3(x), \,q_4(x)$ be the two other quotients that arise from the other factors of $f(x)$ modulo $p$. By Theorem \ref{main}, the following four integral elements are a $p$-integral basis:
$$
q_{1,x-s}(\t)/p^\nu,\ q_2(\t),\ q_3(\t),\ q_4(\t),\quad \nu=\lfloor \Delta/2\rfloor.
$$
Since $q_2(x)$, $q_3(x)$, $q_4(x)$ are monic polynomials with integer coefficients, the following family is a $p$-integral basis too:
\begin{equation}\label{Tdouble}
1,\ \t,\ \t^2,\ (\t^3+s\t^2+(a+s^2)\t+s^3+as+b)/p^\nu, \quad \nu=\lfloor \Delta/2\rfloor.
\end{equation}

\subsection{$f(x)$ has two double roots modulo $p$}\label{Secdubledouble}
Suppose $f(x)$ satisfies (C) of Lemma \ref{factorization}.
We want to choose two regular lifts to $\Z$ of the two double roots modulo $p$.  We show how to achieve this with different arguments for the cases $p>2$ and $p=2$.

\subsubsection{Two double roots, case $p>2$}Recall that $f(x)\equiv (x-s)^2(x+s)^2\md{p}$, 
where $s$ is a square root of $-a/2$ modulo $p$.

\begin{lemma}\label{regreg}
Let us write $b=Bp^m$, $a^2-4c=Ap^{m'}$, for some positive exponents $m,m'$ and some integers $A,B$ not divisible by $p$. Let $r=\min\{m,m'\}$, $\delta=v_p(A^2+8aB^2)$, and let $s$ be a square root of $-a/2$ modulo $p^{r+1}$. Then,
\begin {enumerate}
\item If either $m\ne m'$ or $r\ne\delta$, then $s$ and $-s$ are regular lifts of the two double roots of $f(x)$ modulo $p$. 
\item Suppose $m=m'=\delta$ and $s$ is irregular. Apply to $s$ the iteration method of Lemma \ref{iteration} to obtain a regular $s_n$; then $-s_n$ is regular too. 
\item Suppose $s$ and $-s$ are  regular lifts of the two double roots of $f(x)$ modulo $p$,
and denote $\nu_{\pm}=\lfloor\min\{v_p(f(\pm s))/2,v_p(f'(\pm s))\}\rfloor$. Then,
\end {enumerate}
$$
\min\{\nu_-,\nu_+\}=\left\lfloor r/2\right\rfloor,\quad \max\{\nu_-,\nu_+\}=\left\lfloor (\Delta-r)/2\right\rfloor.
$$
\end{lemma}

\begin{proof}
Assume first $m\ne m'$. From $4f(s)=(2s^2+a)^2+4bs-(a^2-4c)$, $f'(s)=2s(2s^2+a)+b$, we get $v_p(f(\pm s))=r$, $v_p(f'(\pm s))\ge r$, so that $s$ and $-s$ are regular.

Assume now $m=m'$. We have $v_p(f'(\pm s))=r$ in all cases. At least one of the two integers $4Bs-A$, $-4Bs-A$ is not divisible by $p$. Suppose that $s$ is chosen so that $v_p(-4Bs-A)=0$; then $v_p(f(-s))=r$ and $-s$ is regular. On the other hand, 
$v_p(4Bs-A)=v_p(16B^2s^2-A^2)$ is equal to $\delta$, if $\delta\le r$, and it is greater than $r$ otherwise. Hence, $v_p(f(s))=r+\delta$, if $\delta\le r$, and $v_p(f(s))>2r$ otherwise. Thus, $s$ might be irregular only when $\delta=r$. This proves (1).

If $s$ is irregular, we have $v_p(f(s))=2r$, $v_p(f'(s))=r$. By (\ref{grow}), all iterates $s_n$ obtained by the method of Lemma \ref{iteration} satisfy $v_p(f(s_n))>2r$, $v_p(f'(s_n))>r$. Since $f(s_n)-f(-s_n)=2bs_n$, $f'(s_n)+f'(-s_n)=2b$, we get $v_p(f(-s_n))=v_p(f'(-s_n))=r$ and $-s_n$ is regular. This proves (2).

The above arguments show that $\nu_-=\lfloor r/2\rfloor$ (recall that we chose $s$ in such a way that $\nu_-\le\nu_+$).  The ideal $p\Z_K$ splits as the product of two ideals $\mathfrak{a}_s$, $\mathfrak{a}_{-s}$,
both of norm $p^2$, whose factorization is determined by the respective Newton polygons $N_{x-s}^-(f)$, $N_{x+s}^-(f)$. Thus, $v_p(\disc(K))=0,1,2$, according to the number $0,1,2$ of ramified ideals in the pair $\mathfrak{a}_s$, $\mathfrak{a}_{-s}$. On the other hand, $\mathfrak{a}_{-s}$ ramifies if and only if $r$ is odd.
This allows to compute $\nu_+$ form $\nu_-+\nu_+=\ind_p(f)=(\Delta-v_p(\disc(K))/2$.
This proves (3).\end{proof}

Lemma \ref{regreg} shows how to find an integer $s$ such that both $s$ and $-s$ are regular, and in most of the cases $s$ is just a sufficiently good $p$-adic approximation to a square root of $-a/2$.
Theorem \ref{main} provides then a $p$-integral basis:
$$q_{1,x-s}(\t)/p^{\nu_+},\ q_{2,x-s}(\t),\ q_{1,x+s}(\t)/p^{\nu_-},\ q_{2,x+s}(\t),$$where $q_{1,x\pm s}(x)$ are given in (\ref{q1}), and $q_{2,x\pm s}(x)$ are monic polynomials with integer coefficients. By exchanging $s$ and $-s$  we may assume that $\nu_-\le \nu_+$, as we did in the proof of Lemma \ref{regreg}. Then,
$$
(2s\t^2+2s^3+2as)/p^{\nu_-}=(q_{1,x-s}(\t)-q_{1,x+s}(\t))/p^{\nu_-},
$$is integral. Since $p\nmid 2s$, the following family is a $p$-integral basis too.
\begin{equation}\label{Tdoubledoublep}
1,\,\t,\,(\t^2+s^2+a)/p^{\,\nu_-},\, (\t^3+s\t^2+(a+s^2)\t+s^3+as+b)/p^{\nu_+},
\end{equation}
with $\nu_-=\left\lfloor r/2\right\rfloor$, $\nu_+=\left\lfloor(\Delta-r)/2\right\rfloor$.

\subsubsection{Two double roots, case $p=2$} Recall that $f(x)\equiv x^2(x+1)^2\md{2}$, so that any even integer is a lift of the double root $0$ modulo $2$ and any odd integer  is a lift of the double root $1$ modulo $2$.

If $v_2(c)=1$, then $N_x^-(f)$ is one-sided with slope $1/2$, so that $0$ is an even regular lift and $\ind_x(f)=0$. Similarly, if $v_2(a+b+c+1)=1$ then $1$ is an odd regular lift and $\ind_{x+1}(f)=0$. If $v_2(c)>1$ and $v_2(a+b+c+1)>1$, then:
$$
\begin{array}{l} 
a\equiv1\md4\imp v_2(b)=1\imp v_2(f'(0))=1\imp 0\mbox{ is regular},\\
a\equiv-1\md4\imp v_2(b)>1\imp v_2(f'(1))=1\imp 1\mbox{ is regular}.
\end{array}
$$ 
Hence, we apply the iteration method of Lemma \ref{iteration} to obtain a regular lift $s$ of an specific double root of $f(x)$ modulo $2$, and the $2$-integral basis is given by

\begin{small}
\begin{equation}\label{Tdoubledouble2}
\begin{tabular}{|c|c|c|c|}\hline
$v_2(c)$& $v_2(a+b+c+1)$&$s$ &$p$-integral basis\\\hline
$\quad 1$&$\quad 1$&&$1,\,\t,\,\t ^2,\,\t^3$\\\hline
$\quad 1$&$>1$&$s$ odd&$1,\,\t,\,\t ^2,\,\alpha/2^\nu$\\\hline
$>1$&$\quad 1$&$s$ even&$1,\,\t,\,\t ^2,\,\alpha/2^\nu$\\\hline
$>1$&$>1$&$\begin{array}{c}a\equiv1\md4\imp s\mbox{ odd}\\
a\equiv3\md4\imp s\mbox{ even}\end{array}$&$1,\,\t,\,\dfrac{\t^2+\t}2,\, \alpha/2^\nu$
\\\hline
\end{tabular}
\end{equation}\medskip
\end{small}

\noindent where $\alpha=\t^3+s\t^2+(a+s^2)\t+s^3+as+b$, and $\nu=\min\{v_2(f(s))/2,v_2(f'(s))\}$. In fact, if $a\equiv1\md4$, we take $0$ as an even regular lift and we find an odd regular lift $s$. Theorem \ref{main} yields a $2$-integral basis: $q_{1,x}(\t)/2,\ q_{2,x}(\t),\ q_{1,x-s}(\t)/2^\nu,q_{2,x-s}(\t)$, where $q_{1,x}$, $q_{1,x-s}$ are given in (\ref{q1}) and $q_{2,x},q_{2,x-s}$ are  monic polynomials. In particular, $(\t^2+\t)/2$ is integral, because $a$, $s$ are odd and
$$
(s\t^2+s^2\t+s^3+as)/2=(q_{1,x-s}(\t)-q_{1,x}(\t))/2
$$
is integral. Hence,
$1,\,\t,\,(\t^2+\t)/2,\,q_{1,x-s}(\t)/2^\nu$ is a $2$-integral basis too.
The case $a\equiv-1\md4$ follows by symmetric arguments.

\subsection{$f(x)$ has a triple root modulo $p$}
Suppose $f(x)$ satisfies (D) of Lemma \ref{factorization}.
Recall that $f(x)\equiv (x-s)^3(x+3s)\md{p}$, for some integer $s$ not divisible by $p$.
The $(x-s)$-Newton polygon of $f(x)$ has length three and the two first quotients of the $(x-s)$-adic development of $f(x)$ are
$$
q_1(x)=x^3+sx^2+(s^2+a)x+s^3+as+b,\quad q_2(x)=x^2+2sx+3s^2+a.
$$Therefore, if we choose a regular $s$, Theorem \ref{main} yields a $p$-integral basis: 
\begin{equation}\label{triplebasis}
1,\quad\t,\quad\dfrac{\t^2+2s\t+3s^2+a}{p^{\nu_2}},\quad\dfrac{\t^3+s\t^2+(s^2+a)\t+s^3+as+b}{p^{\nu_1}},
\end{equation}
where $\nu_i=\lfloor y_i\rfloor$ and $y_i\in\Q$ is the ordinate of the point on $N_{x-s}^-(f)$ of abscissa $i$, for $i=1,2$. We discuss the choice of a regular $s$ and the computation of $\nu_1,\nu_2$ separately for the cases $p>3$ and $p=3$.  

\subsubsection{Triple root, $p>3$. } The coefficients $a,b,c$ satisfy (\ref{tripleabc}); in particular, $-a/6$ is a square in $\Z_p$. 

\begin{lemma}\label{Ttriple}
Let $s_0$ be an integer satisfying (\ref{tripleabc}) and such that $u_2:=v_p(a+6s_0^2)=v_p(f''(s_0)/2)>\Delta/6$. Denote also
\begin{equation}\label{rescoeffs0}
u_0:=v_p(f(s_0)),\ u_1:=v_p(f'(s_0));\quad\sigma_0:=f(s_0)/p^{u_0},\ \sigma_1:=f'(s_0)/p^{u_1}.
\end{equation}
\begin{enumerate}
\item The integer $s_0$ is irregular if and only if $2u_0=3u_1$ and $p\mid\sigma_1^3+27s_0\sigma_0^2$.
\item If $s_0$ is regular, take $s=s_0$; if $s_0$ is irregular, apply to $s_0$ the iteration  method of Lemma \ref{iteration} to obtain a regular $s$. Then, the $p$-integral basis is given by (\ref{triplebasis}), with
\begin{center}
\begin{small}
\begin{tabular}{|c|c|c|}\hline
$u_0,\,u_1$&$\nu_1$&$\nu_2$\\\hline
$2u_0<3u_1$&$\lfloor 2u_0/3\rfloor$&$\lfloor u_0/3\rfloor$\\\hline
$2u_0\ge 3u_1$&$\lfloor (\Delta-u_1)/2\rfloor$&$\lfloor u_1/2\rfloor$\\\hline
\end{tabular}
\end{small}
\end{center}
\end{enumerate}
\end{lemma}

\begin{proof}
By Theorem \ref{index} we have $\Delta\ge 2\ind_p(f)\ge 2\ind_{x-s_0}(f)$, so that $\,u_2>\ind_{x-s_0}(f)/3$, by the hypothesis. This implies that the point $(2,u_2)$ lies strictly above $N_{x-s_0}^-(f)$, whose shape is given in Figure 6. In fact, the conditions $u_2\le u_0/3$, $u_2\le u_1/2$ lead to a contradiction: $\ind_{x-s_0}(f)=u_2+\lfloor (u_0+u_2)/2\rfloor\ge 3u_2$, or $\ind_{x-s_0}(f)=u_1+u_2\ge 3u_2$. 

\begin{center}
\setlength{\unitlength}{5.mm}
\begin{picture}(16,5)
\put(-.15,2.85){$\bullet$}\put(.85,2.85){$\bullet$}\put(1.85,3.85){$\bullet$}
\put(2.85,.85){$\bullet$}
\put(-1,1){\line(1,0){5}}
\put(0,0.2){\line(0,1){4.4}}
\put(3,1){\line(-3,2){3}}\put(3.02,1){\line(-3,2){3}}
\put(-.4,.4){\begin{footnotesize}$0$\end{footnotesize}}
\put(.9,.4){\begin{footnotesize}$1$\end{footnotesize}}
\put(1.9,.4){\begin{footnotesize}$2$\end{footnotesize}}
\put(2.9,.4){\begin{footnotesize}$3$\end{footnotesize}}
\put(1,.9){\line(0,1){.2}}\put(2,.9){\line(0,1){.2}}
\put(0,-.6){$2u_0<3u_1$}

\put(6.85,2.85){$\bullet$}\put(7.85,2.15){$\bullet$}\put(8.85,3.85){$\bullet$}
\put(9.85,.85){$\bullet$}
\put(6,1){\line(1,0){5}}
\put(7,0.2){\line(0,1){4.4}}
\put(10,1){\line(-3,2){3}}\put(10.02,1){\line(-3,2){3}}
\put(6.6,.4){\begin{footnotesize}$0$\end{footnotesize}}
\put(7.9,.4){\begin{footnotesize}$1$\end{footnotesize}}
\put(8.9,.4){\begin{footnotesize}$2$\end{footnotesize}}
\put(9.9,.4){\begin{footnotesize}$3$\end{footnotesize}}
\put(8,.9){\line(0,1){.2}}\put(9,.9){\line(0,1){.2}}
\put(7,-.6){$2u_0=3u_1$}

\put(13.85,3.85){$\bullet$}\put(14.85,1.85){$\bullet$}
\put(16.85,.85){$\bullet$}\put(15.85,3.85){$\bullet$}
\put(14,0.2){\line(0,1){4.4}}
\put(13,1){\line(1,0){5}}
\put(17,1){\line(-2,1){2}}\put(17,1.02){\line(-2,1){2}}
\put(15,2){\line(-1,2){1}}\put(15,2.02){\line(-1,2){1}}
\put(13.6,.4){\begin{footnotesize}$0$\end{footnotesize}}
\put(14.9,.4){\begin{footnotesize}$1$\end{footnotesize}}
\put(15.9,.4){\begin{footnotesize}$2$\end{footnotesize}}
\put(16.9,.4){\begin{footnotesize}$3$\end{footnotesize}}
\put(15,.9){\line(0,1){.2}}\put(16,.9){\line(0,1){.2}}
\put(14,-.6){$2u_0>3u_1$}
\end{picture}
\end{center}\vskip0.cm
\begin{center}
Figure 6
\end{center}

Suppose $2u_0<3u_1$. The Newton polygon has only one side, with slope $-u_0/3$. If $3\nmid u_0$, the side has degree one; if $3\mid u_0$, the side has degree three and residual polynomial $4\overline{s}_0y^3+\overline{\sigma_0}$, which is always separable in $\F_p[y]$.
Hence, $s_0$ is regular and $\nu_2=\lfloor u_0/3\rfloor$, $\nu_1=\lfloor 2u_0/3\rfloor$. 

Suppose $2u_0>3u_1$. The Newton polygon has two sides, one of them of degree one; the other one has either degree one (if $u_1$ is odd) or it has degree two and residual polynomial $4\bar{s}_0y^2+\overline{\sigma}_1$, which is always separable in $\F_p[y]$.
Hence, $s_0$ is regular and $\nu_2=\lfloor u_1/2\rfloor$, $\nu_1=u_1$. By the Theorem of Ore (Theorem \ref{ore}), $p$ ramifies in $K$ if and only if $u_1$ is odd, in which case $p\Z_K=\p^2\q_1\q_2$ and $v_p(\disc(K)=1$; hence, $\nu_1+\nu_2=\ind_p(f)=\lfloor \Delta/2\rfloor$, and $\nu_1=\lfloor (\Delta-u_1)/2\rfloor$.

Suppose $2u_0=3u_1$. The Newton polygon has only one side, of degree three and residual polynomial $4\overline{s}_0y^3+\overline{\sigma_1}y+\overline{\sigma_0}$. Hence, $s_0$ is regular if and only if $(\overline{\sigma}_1)^3+27\overline{s}_0(\overline{\sigma}_0)^2\ne0$.
If $s_0$ is regular, then $p$ is unramified in $K$ and $\nu_2=u_1/2$, $\nu_1=u_1=(\Delta-u_1)/2$. If $s_0$ is irregular, then all iterates of the process of Lemma \ref{iteration} satisfy $v_p(f''(s_n)/2)=u_1/2$, $v_p(f'(s_n))>u_1$, $v_p(f(s_n))>u_0$ (cf. Remark \ref{remark}); hence $N_{x-s}^-(f)$ will have one side with end points $(2,u_1/2)$, $(3,0)$, so that $\nu_2=u_1/2$. Also, $v_p(\disc(K))=0,1$, because  $p\Z_K=\mathfrak{a}\q_1\q_2$ for some ideal $\mathfrak{a}$ of norm $p^2$, whose factorization depends on the final shape of $N_{x-s}^-(f)$. Hence, $\nu_1=\lfloor (\Delta-u_1)/2\rfloor$.
\end{proof}

\subsubsection{Triple root, $p=3$. }Suppose $p=3$ and $3\mid a$, $3\nmid b$, $3\mid c$. 
If $a\equiv3\md9$, then $-a/6$ is a square in $\Z_3$ and we can argue as in the case $p>3$.

\begin{lemma}\label{Ttriple3}
Suppose $a\equiv3\md9$, $3\nmid b$, $3\mid c$.
Let $s_0$ be an integer such that $u_2:=v_3(f''(s_0)/2)=v_3(a+6s_0^2)>\Delta/6$, and $s_0\equiv -b\md3$. For $i=0,1$, denote also $u_i$, $\sigma_i$ as in (\ref{rescoeffs0}).

\begin{enumerate}
\item The integer $s_0$ is irregular if and only if $2u_0<3u_1$ and $3\mid u_0$. 
\item If $s_0$ is regular, take $s=s_0$; if $s_0$ is irregular, take $s_1=s_0+y3^{u_0/3}$, where $y=1$ if $\sigma_0\equiv b\md3$ and $y=-1$ otherwise. If $s_1$ is regular, take $s=s_1$; if $s_1$ is irregular, apply to $s_1$ the iteration  method of Lemma \ref{iteration} to obtain a regular $s$. Then, the $p$-integral basis is given by (\ref{triplebasis}), with
\end{enumerate}
\end{lemma}

\begin{center}
\begin{small}
\begin{tabular}{|c|c|c|c|c|}\hline
$u_0,u_1$&$v_3(f(s_1))$&$v_3(f'(s_1))$&$\nu_1$&$\nu_2$\\\hline
$2u_0\ge 3u_1$&&&$u_1$&$\lfloor u_1/2\rfloor$\\\hline
$2u_0<3u_1$, $3\nmid u_0$&&&$\lfloor 2u_0/3\rfloor$&$\lfloor u_0/3\rfloor$\\\hline
$2u_0<3u_1$, $3\mid u_0$&$u_0+1$&&$ 2u_0/3$&$ u_0/3$\\\hline
$2u_0<3u_1$, $3\mid u_0$&$>u_0+1$&$(2u_0/3)+1$&$(2u_0/3)+1$&$ u_0/3$\\\hline
$2u_0<3u_1$, $3\mid u_0$&$u_0+2$&$>(2u_0/3)+1$&$(2u_0/3)+1$&$ u_0/3$\\\hline
$2u_0<3u_1$, $3\mid u_0$&$>u_0+2$&$>(2u_0/3)+1$&$\lfloor \Delta/2\rfloor-(u_0/3)-1$&$ (u_0/3)+1$\\\hline
\end{tabular}
\end{small}
\end{center}

\begin{proof}
The proof runs in parallel with that of Lemma \ref{Ttriple}. An essential difference is that now a polynomial of the form $y^3+\sigma$ is always inseparable in $\F_3[y]$, wheras a polynomial of the form $y^3+\sigma y+\tau$ is separable as long as 
$\sigma\in\F_3$ is nonzero.  

If $s_0$ is irregular, we saw in (\ref{iterationthree}) that 
$$
v_3(f(s_1))>u_0,\quad v_3(f'(s_1))>2u_0/3,\quad v_3(f''(s_1))>u_0/3.
$$
If $v_3(f(s_1))\le u_0+2$, or $v_3(f'(s_1))\le(2u_0/3)+1$, then $s_1$ is regular. If $s_1$ is irregular, then $\ind_{x-s_1}(f)\ge \ind_{x-s_0}(f)+3=u_0+3$, because we get at least three more points of integers coordinates below or on the $(x-s_1)$-Newton polygon (two with abscissa $1$ and one with abscissa $2$). We claim that $v_3(f''(s_1))=(u_0/3)+1$ in this case. In fact, $f''(s_1)=f''(s)+24s_0yp^{u_0/3}+12y^2p^{2u_0/3}$; hence,  $v_3(f''(s_1))>(u_0/3)+1$ would imply $v_3(f''(s_0))=(u_0/3)+1$, and we would get  $\Delta\ge 2\ind_p(f)\ge2\ind_{x-s_1}(f)\ge 2u_0+6= 6v_3(f''(s_0))$, in contradiction with our hypothesis. Finally, as in the previous lemma, in this case we have $v_3(\disc(K))=0,1$,
so that $\nu_1+\nu_2=\ind_p(f)=\lfloor\Delta/2\rfloor$. This determines $\nu_1$ as well.
\end{proof}

Suppose now that $a\equiv0,6\md9$. For any integer $s\equiv -b\md3$, we have necessarily $v_3(f''(s))=v_3(a+6s^2)=1$. The following lemma can be proved by completely analogous arguments. 

\begin{lemma}\label{Ttriple3bis}
Suppose $a\equiv0,6\md9$, $3\nmid b$, $3\mid c$. If $v_3(f(-b))\le2$, or $v_3(f'(-b))=1$, then take $s=-b$. If $v_3(f(-b))>2$, and $v_3(f'(-b))>1$, then take  $s$ to be any integer such that $v_3(f'(s))>\Delta/2$, and $s\equiv -b\md3$. Then, $s$ is regular and the $3$-integral basis is given by (\ref{triplebasis}), with
\end{lemma}
\begin{center}
\begin{small}
\begin{tabular}{|c|c|c|c|}\hline
$v_3(f(s))$&$v_3(f'(s))$&$\nu_1$&$\nu_2$\\\hline
$1$&&$0$&$0$\\\hline
$2$&&$1$&$0$\\\hline
$>1$&$1$&$1$&$0$\\\hline
$>2$&$>1$&$\lfloor\Delta/2\rfloor-1$&$1$\\\hline
\end{tabular}
\end{small}
\end{center}

\subsection{$f(x)$ has a $4$-tuple root. }\label{4tuple}Suppose $f(x)$ satisfies (E) of Lemma \ref{factorization}. 

\subsubsection{The three coefficients $a,b,c$ of $f(x)$ are divisible by $p$.} If $v_p(a)\ge 2$, $v_p(b)\ge 3$, and $v_p(c)\ge 4$, then the quartic polynomial $x^4+(a/p^2)x^2+(b/p^3)x+(c/p^4)$ generates the same quartic field. By repeating this procedure, we can assume that either $v_p(a)=1$, or $v_p(b)\le2$, or $v_p(c)\le 3$. Under this assumption, a direct application of Theorem \ref{main} yields a $p$-integral basis in most of the cases. The results are collected in Table \ref{T4tuple}. All rows of this table correspond to a situation where $f(x)$ is $p$-regular, except for the third row  (which reflects a regular situation only if $v_p(a^2-4c)=2$), and the sixth row. These irregular cases cannot be dealt with the iteration method of Lemma \ref{iteration}
because the unique side of the Newton polygon has a non-integer slope (equal to $-1/2$). They are handled in section \ref{secOrder2} by using Newton polygons of second order. 

\begin{center}
\begin{table}
\caption{\footnotesize $f(x)=x^4+ax^2+bx+c$ satisfies (E1) of Lemma \ref{factorization}. In  rows $7$ and $8$, we use a regular integer $s$ obtained by the iteration method of Lemma \ref{iteration}, and we denote $\alpha=\t^3+s\t^2+(a+s^2)\t+s^3+as+b$, $u_0=v_2(f(s))$, $u_1=v_2(f'(s))$.}\label{T4tuple}
\begin{small}
\begin{tabular}{|c|c|c|c|c|c|}\hline
$v_p(c)$&$v_p(b)$&$v_p(a)$&$p$&$\nu$&$p$-integral basis\\\hline
$1$&&&&&$1,\,\t,\,\t^2,\,\t^3$\\\hline
$>1$&$1$&&&&$1,\,\t,\,\t^2,\,\t^3/p$\\\hline
$2$&$>1$&$1$&$>2$&$\frac12\min\{v_p(b),v_p(a^2-4c)\}$&$1,\t,\dfrac{\t^2+(a/2)}{p^{\lfloor\nu\rfloor}},\dfrac{\t^3+(a/2)\t}{p^{\lfloor\nu+(1/2)\rfloor}}$\\\hline
$2$&$>1$&$1$&$2$&&$1,\,\t,\,\t^2/2,\,\t^3/2$\\\hline
$2$&$>1$&$>1$&$>2$&&$1,\,\t,\,\t^2/p,\,\t^3/p$\\\hline
$2$&$>1$&$>1$&$2$&Table \ref{T4tuple2irreg}&$1,\,\t,\,\dfrac{Q(\t)}{2^{\lfloor\nu\rfloor}},\,\dfrac{\t Q(\t)}{2^{\lfloor\nu+(1/2)\rfloor}}$\\\hline
$>2$&$>1$&$1$&$>2$&$\lfloor\Delta/2\rfloor-1$&$1,\,\t,\,\t^2/p,\,\alpha/p^\nu$\\\hline
$>2$&$>1$&$1$&$2$&$\lfloor\min\{\frac12(u_0+1),u_1\}\rfloor$&$1,\,\t,\,\t^2/2,\,\alpha/2^\nu$\\\hline
$>2$&$2$&$>1$&&&$1,\,\t,\,\t^2/p,\,\t^3/p^2$\\\hline
$3$&$>2$&$>1$&&&$1,\,\t,\,\t^2/p,\,\t^3/p^2$\\\hline
\end{tabular}
\end{small}
\end{table}
\end{center}

\begin{center}
\begin{table}
\caption{\footnotesize Expansion of the sixth row of Table \ref{T4tuple}.  }\label{T4tuple2irreg}
\begin{small}
\begin{tabular}{|c|c|c|c|c|c|}\hline
$v_2(a)$&$v_2(b)$&$v_2(2a+c-4)$&$v_2(2a+b)$&$Q(x)$&$\nu$\\\hline
&$2$&&&$x^2$&$5/4$\\\hline
&$3$&$3$&&$x^2+2$&$7/4$\\\hline
$2$&$3$&$\ge4$&&$x^2+2x+2$&$2$\\\hline
$\ge3$&$3$&$\ge4$&&$x^2+2$&$7/4$\\\hline
$2$&$\ge4$&$\ge4$&&$x^2+2$&$7/4$\\\hline
$2$&$\ge4$&$3$&&see below&see below
\\\hline
$\ge3$&$\ge4$&$3$&&$x^2+2$&$2$\\\hline
$\ge3$&$\ge4$&$4$&$>4$&$x^2+2x+2$&$9/4$\\\hline
$\ge3$&$\ge4$&$4$&$4$&$x^2+2x+2$&$5/2$\\\hline
$\ge3$&$\ge4$&$\ge5$&$4$&$x^2+2x+2$&$9/4$\\\hline
$\ge3$&$\ge4$&$\ge5$&$>4$&$x^2+2x+2$&$5/2$\\\hline
\end{tabular}

If $v_2(a)=2$, $v_2(b)\ge4$ and $v_2(2a+c-4)=3$, we denote $u=v_2(b),\,v=v_2(c-(a^2/4))$, $d=(c-(a^2/4))/2^v \md4$.
\begin{tabular}{|c|c|c|c|}\hline
$u,\,v$&$d$&$Q(x)$&$\nu$\\\hline
$u\le v$&&$x^2+\frac a2$&$\frac{2u+1}4$\\\hline
$u-1=v=2w$&$-1$&$x^2+\frac a2+2^w$&$w+\frac34$\\\hline
$u-1>v=2w$&$-1$&$x^2+\frac a2+2^w$&$w+1$\\\hline
$u-1=v=2w$&$1$&$x^2+2^wx+\frac a2+2^w$&$w+1$\\\hline
$u-1>v=2w$&$1$&$x^2+\frac a2+2^w$&$w+\frac34$\\\hline
$u-1=v=2w+1$&$\frac a4$&$x^2+2^wx+\frac a2$&$w+\frac54$\\\hline
$u-1>v=2w+1$&$\frac a4$&$x^2+2^wx+\frac a2$&$w+\frac32$\\\hline
$u-1=v=2w+1$&$-\frac a4$&$x^2+2^wx+\frac a2+2^{w+1}$&$w+\frac32$\\\hline
$u-1>v=2w+1$&$-\frac a4$&$x^2+2^wx+\frac a2$&$w+\frac54$\\\hline
\end{tabular}
\end{small}
\end{table}
\end{center}

\subsubsection{The case $p=2$, $a$, $b$ even and $c$ odd. }
We have $f(x)\equiv (x+1)^4\md2$. For any odd integer $m$, the $(x-m)$-development of $f(x)$ is:
$$
f(x)=(x-m)^4+4m(x-m)^3+A(x-m)^2+B(x-m)+C, 
$$with $A=6m^2+a,\ B=f'(m),\ C=f(m)$. The polynomial
$$
g(x):=f(x+m)=x^4+4mx^3+Ax^2+Bx+C,
$$
defines the same quartic field $K$, and since $A,B,C$ are all even, the computation of a $2$-integral basis is similar to the previous case. We take $\om:=\t-m$ as a root of $g(x)$, and we display the results in Table \ref{T4extra}, whose rows take into account all possible values of the triple $(v_2(A),v_2(B),v_2(C))$, except for three cases that may be discarded by a proper choice of $m$. Lemma \ref{avoid}, whose proof is left to the reader, shows how to avoid these three bad cases.

\begin{lemma}\label{avoid}
Let $m$, $m'$ be odd integers and let $f(x+m)=x^4+4mx^3+Ax^2+Bx+C$, $f(x+m')=x^4+4m'x^3+A'x^2+B'x+C'$.

If $v_2(A)>2$, $v_2(B)>3$ and $v_2(C)=4$, then for $m'=m+2$, we get $v_2(A')>2$, $v_2(B')>3$ and $v_2(C')>4$.

If $v_2(A)>4$, $v_2(B)=6$ and $v_2(C)=8$, then for $m'=m+4$, we get $v_2(A')=4$, $v_2(B')>6$,  $v_2(C')>8$.

If $v_2(A)=4$, $v_2(B)=6$ and $v_2(C)>8$, then for $m'=m+4$, we get $v_2(A')>4$, $v_2(B')>6$ and $v_2(C')>8$.
\end{lemma}

Most of the rows of Table \ref{T4extra} correspond to a situation where $g(x)$ is $2$-regular, with some exceptions that we discuss now.

In rows $5$, $18$ and $24$, regularity is achieved by applying the iteration method of Lemma \ref{iteration}.
The cases of rows $4$, $16$ and $17$ are handled in section \ref{secOrder2} by using Newton polygons of second order. 

In rows $10$ and $11$, we consider the polynomial $h(x)=g(2x)/16=x^4+2x^3+A'x^2+B'x+C'$, where $A'=A/4$, $B'=B/8$, $C'=C/16$. This polynomial is another defining equation of our quartic field $K$. 

If $v_2(A)=2$, $v_2(B)>3$, $v_2(C)=4$ (row $10$), then $h(x)\equiv (x^2+x+1)^2\md2$, and the discussion is similar to that of section \ref{Secquadratic}. In the table below, we display a choice of $\phi(x)\in\Z[x]$ such that $\phi(x)\equiv x^2+x+1\md2$ and $h(x)$ is $\phi$-regular. We denote $u=v_2(B'+1-A')$, $v=v_2(C'-((A'-1)^2/4))$ in this table.\medskip

\begin{center}
\begin{small}
\begin{tabular}{|c|c|c|}\hline
$A'\md4$&$u,\,v$&$\phi(x)$\\\hline
$1$&&$x^2+x-1$\\\hline
$-1$&$\min\{u,v\}$ odd&$x^2+x+\frac {A'-1}2$\\\hline
$-1$&$u=v=2w$&$x^2+(1+2^w)x+\frac {A'-1}2$\\\hline
$-1$&$u=2w<v$&$x^2+(1+2^w)x+\frac {A'-1}2+2^w$\\\hline
$-1$&$u>v=2w$&$x^2+x+\frac {A'-1}2+2^w$\\\hline
\end{tabular}
\end{small}
\end{center}\medskip

We display the $2$-integral basis provided by 
Theorem \ref{main} in Table \ref{T4extrairreg1}.

If $v_2(A)=2$, $v_2(B)>3$, $v_2(C)>4$ (row $11$), then $h(x)\equiv x^2(x+1)^2\md2$, and the discussion is similar to that of section \ref{Secdubledouble}. 
We display the results in Table \ref{Tdoubledoublebis}. For any integer $s$ the first quotient of the $(x-s)$-development of $h(x)$, evaluated at the root $\tau:=\om/2$ of $h(x)$ is:
\begin{equation}\label{alphasbis}
\alpha_s:=\tau^3+(s+2)\tau^2+(s^2+2s+A')\tau+s^3+2s^2+A's+B'.
\end{equation}
If $s$ is regular, the ordinate of the point on $N_{x-s}^-(h)$ of abscissa one is:
\begin{equation}\label{nusbis}
\nu_s:=\left\lfloor\min\left\{\frac12v_2(h(s)),v_2(h'(s))\right\}\right\rfloor
\end{equation}

Finally, in row $25$, we consider the polynomial $h(x)=g(4x)/64=x^4+mx^3+A'x^2+B'x+C'$, where $A'=A/16$, $B'=B/64$, $C'=C/256$. This polynomial defines the same quartic field $K$, and it satisfies $h(x)\equiv x^3(x+1)\md2$; we take $\tau:=\om/4$ as a root of $h(x)$ in $K$. A $2$-integral basis is obtained by the computation of a regular lift $s$ of the triple root modulo $2$, by the iteration method of Lemma \ref{iteration}. Theorem \ref{main} yields:
\begin{multline}\label{last}
1,\,\tau,\,\dfrac{\tau^2+(2s+m)\tau+3s^2+2ms+A'}{2^{\nu_2}},\,\\
\dfrac{\tau^3+(s+m)\tau^2+(s^2+ms+A')\tau+s^3+ms^2+A's+B'}{2^{\nu_1}}, 
\end{multline}
where $\nu_2=\min\left\{\frac13v_2(h(s)),\frac12v_2(h'(s)),v_2(\frac12h''(s))\right\}$is the ordinate of the point on $N_{x-s}^-(h)$ of abscissa two, and $\nu_1=\min\left\{\frac12(v_2(h(s))+\nu_2),v_2(h'(s))\right\}$ that of the point of abscissa one.

\begin{center}
\begin{table}
\caption{\footnotesize $f(x)$ satisfies (E2) of Lemma \ref{factorization} and $g(x)=f(x+m)=x^4+4mx^3+Ax^2+Bx+C$. In the first twelve rows we take $m=1$; in the rest of the rows we consider other odd values of $m$ to avoid some bad cases (see Lemma \ref{avoid}). In rows $5$, $18$, $24$, we use a regular integer $s$ obtained by the iteration method of Lemma \ref{iteration}; we denote $\alpha=\om^3+(s+4m)\om^2+(A+4ms+s^2)\om+s^3+4ms^2+As+B$, and $\nu=\left\lfloor\min\{(v_2(g(s))+v_2(A))/2,v_2(g'(s))\}\right\rfloor$. }\label{T4extra}
\begin{small}
\begin{tabular}{|c|c|c|c|c|}\hline
$v_2(C)$&$v_2(B)$&$v_2(A)$&$2$-integral basis\\\hline
$1$&&&$1,\,\om,\,\om^2,\,\om^3$\\\hline
$>1$&$1$&&$1,\,\om,\,\om^2,\,\om^3/2$\\\hline
$2$&$>1$&$1$&$1,\,\om,\,\om^2/2,\,\om^3/2$\\\hline
$2$&$>1$&$>1$&Table \ref{T4extrairreg}\\\hline
$>2$&$>1$&$1$&$1,\,\om,\,\om^2/2,\,\alpha/2^\nu$\\\hline
$>2$&$2$&$>1$&$1,\,\om,\,\om^2/2,\,\om^3/4$\\\hline
$3$&$>2$&$>1$&$1,\,\om,\,\om^2/2,\,\om^3/4$\\\hline
$4$&$3$&$2$&$1,\,\om/2,\,\om^2/4,\,\om^3/8$\\\hline
$4$&$3$&$>2$&$1,\,\om/2,\,\om^2/4,\,\om^3/8$\\\hline
$4$&$>3$&$2$&Table \ref{T4extrairreg1}\\\hline
$>4$&$>3$&$2$&Table \ref{Tdoubledoublebis}\\\hline
$>4$&$3$&$\ge2$&$1,\,\om/2,\,\om^2/4,\,\om^3/8$\\\hline
$5$&$>3$&$>2$&$1,\,\om/2,\,\om^2/4,\,\om^3/8$\\\hline
$>5$&$4$&$>2$&$1,\,\om/2,\,\om^2/4,\,\om^3/16$\\\hline
$6$&$>4$&$3$&$1,\,\om/2,\,\om^2/8,\,\om^3/16$\\\hline
$6$&$5$&$\ge4$&$1,\,\om/2,\,(\om^2+8)/8,\,(\om^3+8\om)/32$ \\\hline
$6$&$>5$&$\ge4$&$1,\,\om/2,\,\om^2/8,\,\om^3/16$ \\\hline
$>6$&$>4$&$3$&$1,\,\om/2,\,\om^2/8,\,\alpha/2^\nu$\\\hline
$>6$&$5$&$\ge4$&$1,\,\om/2,\,\om^2/8,\,\om^3/32$\\\hline
$7$&$>5$&$\ge4$&$1,\,\om/2,\,\om^2/8,\,\om^3/32$\\\hline
$8$&$6$&$4$&$1,\,\om/4,\,\om^2/16,\,\om^3/64$\\\hline
$8$&$>6$&$\ge4$&$1,\,\om/4,\,\om^2/16,\,\om^3/64$\\\hline
$>8$&$6$&$>4$&$1,\,\om/4,\,\om^2/16,\,\om^3/64$\\\hline
$>8$&$>6$&$4$&$1,\,\om/4,\,\om^2/16,\,\alpha/2^\nu$\\\hline
$>8$&$>6$&$>4$&(\ref{last})\\\hline
\end{tabular}
\end{small}
\end{table}
\end{center}

\begin{center}
\begin{table}
\caption{\footnotesize Expansion of the fifth row of Table \ref{T4extra}.  The $2$-integral basis is $1,\,\om,\,Q(\om)/2^{\lfloor\nu\rfloor},\,Q(\om)/2^{\lfloor\nu+(1/2)\rfloor}$. }\label{T4extrairreg}
\begin{small}
\begin{tabular}{|c|c|c|c|c|}\hline
$v_2(A)$&$v_2(B+8)$&$v_2(2A+C+4)$&$Q(x)$&$\nu$\\\hline
&$2$&&$x^2$&$5/4$\\\hline
&$3$&$\ge4$&$x^2+2$&$7/4$\\\hline
$2$&$3$&$3$&$x^2+2x+2$&$2$\\\hline
$2$&$\ge4$&$3$&$x^2+2$&$7/4$\\\hline
$2$&$4$&$\ge4$&$x^2+2$&$9/4$\\\hline
$2$&$\ge5$&$\ge5$&$x^2-2$&$5/2$\\\hline
$2$&$\ge5$&$4$&$x^2+2$&$5/2$\\\hline
$\ge3$&$3$&$3$&$x^2+2$&$7/4$\\\hline
$\ge3$&$\ge4$&$3$&see below&see below
\\\hline
$\ge3$&$\ge4$&$\ge4$&$x^2+2$&$2$\\\hline
\end{tabular}

If $v_2(A)\ge3$, $v_2(B+8)\ge4$ and $v_2(2A+C+4)=3$, we denote $u=v_2(B+8-2A)$, $v=v_2(C-((A-4)^2/4))$, $d=(C-((A-4)^2/4))/2^v\md4$, $e=(B+8-2A)/2^u\md4$. 
\begin{tabular}{|c|c|c|c|c|}\hline
$u,\,v$&$d$&$e$&$Q(x)$&$\nu$\\\hline
$u<v,\,u=v=2w$&&&$x^2+2x-2+\frac A2$&$\frac{2u+1}4$\\\hline
$u=v=2w+1$&$1+\frac A4$&$1$&$x^2+(2+2^w)x-2+\frac A2$&$w+\frac54$\\\hline
$u=v=2w+1$&$1+\frac A4$&$-1$&$x^2+(2+2^w)x-2+\frac A2+2^{w+1}$&$w+\frac32$\\\hline
$u=v=2w+1$&$-1+\frac A4$&$-1$&$x^2+(2+2^w)x-2+\frac A2$&$w+\frac54$\\\hline
$u=v=2w+1$&$-1+\frac A4$&$1$&$x^2+(2+2^w)x-2+\frac A2$&$w+\frac32$\\\hline
$u-1=v=2w$&&&$x^2+2x-2+\frac A2+2^w$&$w+\frac34$\\\hline
$u-1>v=2w$&$-1$&&$x^2+2x-2+\frac A2+2^w$&$w+1$\\\hline
$u-1>v=2w$&$1$&&$x^2+(2+2^w)x-2+\frac A2+2^w$&$w+1$\\\hline
$u>v=2w+1$&&&$x^2+2x-2+\frac A2$&$w+\frac34$\\\hline
\end{tabular}
\end{small}
\end{table}
\end{center}

\begin{center}
\begin{table}
\caption{\footnotesize Expansion of the tenth row of Table \ref{T4extra}. We take $h(x)=g(2x)/16=x^4+2x^3+A'x^2+B'x+C'$, with $A',C'$ odd, $B'$ even. We denote $u=v_2(B'+1-A')$, $v=v_2(C'-((A'-1)^2/4))$, $r=\min\{u,v\}$, $e=(B'+1-A')/2^u\md4$, $d=(C'-((A'-1)^2/4))/2^v\md4$, $k=A'+2^u\md8$. The $2$-integral basis is $1,\,\tau,\,Q(\tau)/2^\nu,\,\tau Q(\tau)/2^\nu$, where $\tau=\om/2$ is a root of $h(x)$. }\label{T4extrairreg1}
\begin{small}

If $A'\equiv1\md4$

\begin{tabular}{|c|c|c|}\hline
&$Q(x)$&$\nu$\\\hline
$C'\equiv1\md4$ or $v_2(B')=1$&$x^2$&$0$\\\hline
otherwise&$x^2+x+1$&$1$\\\hline
\end{tabular}\bigskip

If $A'\equiv3\md4$
\begin{tabular}{|c|c|c|c|}\hline
&&$Q(x)$&$\nu$\\\hline
$r$\mbox{ odd}&&$x^2+x+\frac {A'-1}2$&$\left\lfloor\frac r2\right\rfloor$\\\hline
$u>v=2w$&$u=2w+1$ or $d=1$&$x^2+x+\frac {A'-1}2$&$w$\\\hline
$u>v=2w$&otherwise&$x^2+x+\frac {A'-1}2+2^w$&$w+1$\\\hline
$u=v=2w$&\as{1}$\begin{array}{c}d=\frac{1-A'}2,\mbox{ or }(w=1,e=1),\\\mbox{ or } (w>1,e=-1)\end{array}$&$x^2+x+\frac {A'-1}2$&$w$\\\hline
$u=v=2w$&otherwise&$x^2+(1+2^w)x+\frac {A'-1}2$&$w+1$\\\hline
$u=2w<v$&\as{1}$\begin{array}{c}e=1+2^w,\mbox{ or }\\(v=2w+1,k=3),\\\mbox{ or } (v>2w+1,k=7)\end{array}$&$x^2+x+\frac {A'-1}2$&$w$\\\hline
$u=2w<v$&otherwise&$x^2+(1+2^w)x+\frac {A'-1}2+2^w$&$w+1$\\\hline
\end{tabular}
\end{small}
\end{table}
\end{center}

\begin{center}
\begin{table}
\caption{\footnotesize Expansion of the eleventh row of Table \ref{T4extra}. We take $h(x)=g(2x)/16=x^4+2x^3+A'x^2+B'x+C'$, with $A'$ odd, $B',\,C'$ even. We denote $\tau=\om/2$, where $\om$ is a root of $g(x)$. The integers $s,t$ are regular lifts of the two double roots of $h(x)$ modulo $2$, obtained by the iteration method of Lemma \ref{iteration}, and their notation is chosen so that $\nu_t\le \nu_s$.  The data $\nu_s$, $\alpha_s$ are defined in (\ref{alphasbis}), (\ref{nusbis}), and $\beta_{s,t}=\tau^2+(s+t+2)\tau+s^2+st+t^2+2(s+t)+A'$.}\label{Tdoubledoublebis}
\begin{small}
\begin{tabular}{|c|c|c|c|}\hline
$v_2(C')$& $v_2(A'+B'+C'+3)$&$A',\,t,\,s$ &$p$-integral basis\\\hline
$\quad 1$&$\quad 1$&&$1,\,\tau,\,\tau ^2,\,\tau^3$\\\hline
$\quad 1$&$>1$&$t=0$, $s$ odd&$1,\,\tau,\,\tau^2,\,\alpha_s/2^{\nu_s}$\\\hline
$>1$&$1$&$t=1$, $s$ even&$1,\,\tau,\,\tau^2,\,\alpha_s/2^{\nu_s}$\\\hline
$>1$&$>1$&\as{1}$\begin{array}{c}
A'\equiv3\md4\\t=0,\, s\mbox{ odd}               \end{array}$&$1,\,\tau,\,\dfrac{\tau^2+\tau}2,\, \alpha_s/2^{\nu_s}$\\\hline
$>1$&$>1$&$A'\equiv1\md4$&$1,\,\tau,\,\beta_{s,t}/2^{\nu_t},\, \alpha_s/2^{\nu_s}$
\\\hline
\end{tabular}
\end{small}
\end{table}
\end{center}

\section{Newton polygons of second order}\label{secOrder2}
The theory of Newton polygons of higher order was developed in \cite{m} (and revised in \cite{GMN}) as a tool to factorize separable polynomials in $\Z_p[x]$. It conjecturally yields a fast algorithm to compute integral basis in number fields \cite{GMN2}. In this section we shall use second order polygons directly addressed to cover the few cases of section \ref{secComputation} where a quartic polynomial is not $p$-regular
and the iteration method of Lemma \ref{iteration} cannot be applied because the irregular side of the Newton polygon has non-integer slope. 
In these cases we can find a $p$-integral basis by generalizing Theorem \ref{main} to certain Newton polygons of second order. 

\subsection{Second order $p$-integral bases} Let $F(x)=x^4+a_3x^3+a_2x^2+a_1x+a_0\in\Z[x]$ be an irreducible polynomial; let $\t\in\qb$ be a root of $F(x)$, and $K$ the quartic field generated by $\t$. Let $p$ be a prime number and denote
$$
u_i=v_p(a_i), \quad  \sigma_i=a_i/p^{u_i}, \quad i=0,1,2,3.
$$
We assume throughout this section that $u_0=2,\, u_1>1, \, u_2\ge1, \, u_3\ge1$,  and $R(y):=y^2+\overline{\sigma}_2y+\overline{\sigma}_0\in\F_p[y]$ is inseparable. 
Hence, $N_x(F)$ has only one side, with length four, degree two and slope $-1/2$ (see Figure 7).

\begin{center}
\setlength{\unitlength}{5.mm}
\begin{picture}(5,4)
\put(-.15,1.85){$\bullet$}\put(1.85,.85){$\bullet$}\put(1.85,1.85){$\bullet$}
\put(3.85,-.15){$\bullet$}\put(2,2){\vector(0,1){1}}
\put(-1,0){\line(1,0){5}}
\put(0,-.6){\line(0,1){4}}
\put(4,0){\line(-2,1){4}}\put(4.02,0){\line(-2,1){4}}
\put(-.4,-.6){\begin{footnotesize}$0$\end{footnotesize}}
\put(.9,-.6){\begin{footnotesize}$1$\end{footnotesize}}
\put(1.9,-.6){\begin{footnotesize}$2$\end{footnotesize}}
\put(2.9,-.6){\begin{footnotesize}$3$\end{footnotesize}}
\put(3.9,-.6){\begin{footnotesize}$4$\end{footnotesize}}
\put(1,-.1){\line(0,1){.2}}\put(2,-.1){\line(0,1){.2}}
\put(3,-.1){\line(0,1){.2}}\put(-.1,1){\line(1,0){.2}}
\put(-.6,1.9){\begin{footnotesize}$2$\end{footnotesize}}
\put(-.6,.9){\begin{footnotesize}$1$\end{footnotesize}}
\put(2.6,.9){\begin{footnotesize}$(p>2)$\end{footnotesize}}
\put(2.6,1.9){\begin{footnotesize}$(p=2)$\end{footnotesize}}
\end{picture}
\end{center}\medskip
\begin {center}
Figure 7
\end {center}\be

Choose a $p$-local integer $y\in \zp$ whose reduction modulo $p$ is the double root of $R(y)$; for instance, take $y=1$ if $p=2$. Choose an arbitrary $z\in\zp$ and consider the monic polynomial $\phi(x)=x^2+zpx-yp\in\Z_{(p)}[x]$.
Let $\vp$ be the \emph{$p$-adic valuation of second order} of $\Q_p(x)$ defined in \cite[Sec.2.2]{GMN}. We recall the following properties of $\vp$:

\begin{enumerate}
 \item $\vp(m)=2v_p(m)$, for all $m\in\Z_p$,
\item $\vp(x)=1$, $\vp(\phi(x))=2$,
\item $\vp(mx+n)=\min\{\vp(mx),\vp(n)\}$, for all $m,n\in\Z_p$.
\end{enumerate}

Let $F(x)=\phi(x)^2+a_1(x)\phi(x)+a_0(x)$ be the $\phi$-adic development of $F(x)$. The \emph{$\phi$-Newton polygon of the second order}, $\n(F)$, 
is by definition the lower convex envelope of the set of points $(i,\vp\left(a_i(x)\phi(x)^i\right))$, $i=0,1,2$, of the Euclidean plane. The possible shapes of this polygon are displayed in Figure 8.

\begin{center}
\setlength{\unitlength}{5.mm}
\begin{picture}(18,6)
\put(-.15,3.85){$\bullet$}\put(3.85,1.85){$\bullet$}
\put(1.85,2.85){$\bullet$}\put(1.85,3.85){$\bullet$}
\put(2,4){\vector(0,1){1}}
\put(-1,0){\line(1,0){6}}\put(0,2){\line(0,1){4}}
\multiput(0,-.2)(0,.25){9}{\vrule height2pt}
\multiput(4,-.1)(0,.25){9}{\vrule height2pt}
\put(4,2){\line(-2,1){4}}\put(4.02,2){\line(-2,1){4}}
\put(-.4,-.6){\begin{footnotesize}$0$\end{footnotesize}}
\put(1.9,-.6){\begin{footnotesize}$1$\end{footnotesize}}
\put(3.9,-.6){\begin{footnotesize}$2$\end{footnotesize}}
\put(2,-.1){\line(0,1){.2}}\put(4,-.1){\line(0,1){.2}}
\put(-.1,2){\line(1,0){.2}}
\put(-3.2,3.9){\begin{footnotesize}$\vp(a_0(x))$\end{footnotesize}}
\put(-.6,1.9){\begin{footnotesize}$4$\end{footnotesize}}
\put(-1.4,-1.4){\begin{footnotesize}$\vp(a_0(x))\le 2\vp(a_1(x))$\end{footnotesize}}

\put(12.85,4.85){$\bullet$}\put(16.85,1.85){$\bullet$}\put(14.85,2.85){$\bullet$}
\put(12,0){\line(1,0){6}}
\put(13,2){\line(0,1){4}}
\multiput(13,-.2)(0,.25){9}{\vrule height2pt}
\multiput(17,-.1)(0,.25){9}{\vrule height2pt}
\multiput(13,3)(.25,0){9}{\hbox to 2pt{\hrulefill }}
\put(17,2){\line(-2,1){2}}\put(17.02,2){\line(-2,1){2}}
\put(15,3){\line(-1,1){2}}\put(15.02,3){\line(-1,1){2}}
\put(12.6,-.6){\begin{footnotesize}$0$\end{footnotesize}}
\put(14.9,-.6){\begin{footnotesize}$1$\end{footnotesize}}
\put(16.9,-.6){\begin{footnotesize}$2$\end{footnotesize}}
\put(15,-.1){\line(0,1){.2}}
\put(12.9,2.05){\line(1,0){.2}}\put(12.9,3.05){\line(1,0){.2}}
\put(9.7,4.9){\begin{footnotesize}$\vp(a_0(x))$\end{footnotesize}}
\put(8.7,2.9){\begin{footnotesize}$\vp(a_1(x))+2$\end{footnotesize}}
\put(12.4,1.9){\begin{footnotesize}$4$\end{footnotesize}}
\put(11.7,-1.4){\begin{footnotesize}$\vp(a_0(x))> 2\vp(a_1(x))$\end{footnotesize}}
\end{picture}
\end{center}\be
\begin {center}
Figure 8
\end {center}\be

We define the \emph{second order index} $\id(F)$ to be the number of points of integer coordinates below or on $\n(F)$, strictly above the horizontal line $x=4$
and strictly beyond the vertical axis. In other words, if $Y\in\Q$ is the ordinate of the point on $\n(F)$ of abscissa one, we have $\id(F)=\lfloor Y-4\rfloor$. 
  
There is a natural residual polynomial of second order attached to each side, whose degree coincides with the degree of the side \cite[Sec.2.5]{GMN}. 
Only the points lying on the side determine a nonzero coefficient of this second order residual polynomial, and this is the only property we need in what follows.   
We define $F(x)$ to be \emph{$\phi$-regular in second order} when all  second order residual polynomials are separable.
In the regular case, the $p$-valuation of the index of $F(x)$ can be computed in terms of the higher order indices \cite[Thm.4.18]{GMN}. In our situation this theorem states that:
$$
\ind_p(F)=\ind_x(F)+\id(F)=2+\id(F)=\lfloor Y\rfloor -2.
$$

\begin{theorem}\label{main2}
With the above notations, let $Q(x)=\phi(x)+a_1(x)\in\Z_{(p)}[x]$ be the first quotient of the $\phi$-adic development of $F(x)$. If $F(x)$ is $\phi$-regular in second order, then the following elements are a $p$-integral basis of the quartic field $K$:
\begin{equation}\label{basisnu}
1,\ \t,\ \dfrac{Q(\t)}{p^{\lfloor \nu\rfloor}}, 
\ \dfrac{\t Q(\t)}{p^{\lfloor \nu+(1/2)\rfloor}},\quad \nu:=\dfrac Y2-1.
\end{equation}
\end{theorem}

\begin{proof}
Arguing as in Proposition \ref{denominator}, one checks easily that 
$v_{\p}(Q(\t))\ge e(\p/p)\nu$, 
for all prime ideals $\p$ of $K$ lying above $p$. This is proved for Newton polygons of arbitrary order in \cite[Prop.3.6]{GMN2}. 
On the other hand, along the proof of Proposition \ref{denominator} we saw that $v_{\p}(\t)\ge e(\p/p)/2$, for all prime ideals $\p$ of $K$ lying above $p$. Therefore, the four elements of (\ref{basisnu})
are integral, and they
 generate a $\Z_{(p)}$-module $M$ containing $\Z_{(p)}[\t]$, with
$$
v_p\left(\left(M\colon \Z_{(p)}[\t]\right)\right)=\left\lfloor \dfrac Y2\right\rfloor+\left\lfloor \dfrac{Y+1}2\right\rfloor-2=\lfloor Y\rfloor-2=\ind_p(f).
$$
\end{proof}

In the rest of the section we apply Theorem \ref{main2} to compute a $p$-integral basis in four 
subcases of \S \ref{4tuple}. Actually, in the discussion of some of these cases we just show how to choose a polynomial $\phi(x)$ such that $F(x)$ is $\phi$-regular in second order. 

\subsection{Case $p>2$, $F(x)=f(x)=x^4+ax^2+bx+c$, $v_p(a)=1$, $v_p(b)>1$, $v_p(c)=2$. }This case corresponds to row $3$ of Table \ref{T4tuple}.
The polynomial $f(x)$ is $x$-regular if and only if $v_p(a^2-4c)=2$, in which case, $1,\,\t,\,\t^2/p,\,\t^3/p$ is a $p$-integral basis. In the irregular case,  $v_p(a^2-4c)\ge 3$, we choose $\phi(x)=x^2+(a/2)$, leading to a $\phi$-adic development:  $f(x)=\phi(x)^2+bx+c-(a^2/4)$. Denote $u=\vp(bx+c-(a^2/4))=\min\{2v_p(b)+1,2v_p(c-(a^2/4))\}$. 
The $\phi$-Newton polygon of second order has only one side, with end points $(0,u)$, $(2,4)$. This side has degree one if $v_p(b)<v_p(a^2-4c)$ (because then $u$ is odd), and  degree two otherwise; but in the latter case the residual polynomial is a separable polynomial of the form $\sigma y^2+\tau$ for some nonzero $\sigma,\tau\in\F_p$. Thus, $f(x)$ is always $\phi$-regular in second order. We have $Y=(u+4)/2$, and $\nu=u/4$; however, in Table \ref{T4tuple} we displayed $\nu=\frac 12\min\{v_p(b),v_p(a^2-4c)\}$ because it yields too the right values of $\lfloor \nu\rfloor$, $\lfloor \nu+(1/2)\rfloor$. 

\subsection{Case $p=2$, $F(x)=f(x)=x^4+ax^2+bx+c$, $v_2(a)>1$, $v_2(b)>1$, $v_2(c)=2$. }This case corresponds to row $6$ of Table \ref{T4tuple}.
For the different choices of $\phi(x)$ specified in Table \ref{ChoiceofPhi} the polynomial $f(x)$ is $\phi$-regular in second order.
We discuss in some detail only the cases covered by $\phi(x)=x^2-2$. The $\phi$-adic development of $f(x)$, and the first quotient of this development are:
$$
f(x)=\phi(x)^2+(a+4)\phi(x)+bx+c+2a+4,\quad Q(x)=\phi(x)+a+4=x^2+a+2.
$$ 
If $v_2(a)=2$, $v_2(b)\ge4$ and $v_2(c+2a+4)\ge4$, the $\phi$-Newton polygon of second order of $f(x)$ can have different shapes; in Figure 9 we display $\n(f)$ in all other cases.

\begin{center}
\setlength{\unitlength}{5.mm}
\begin{picture}(21,6.6)
\put(-.15,1.85){$\bullet$}
\put(.85,2.85){$\bullet$}\put(1.85,.85){$\bullet$}
\put(-1,0){\line(1,0){4}}
\put(0,1){\line(0,1){5}}
\put(1,3){\vector(0,1){1}}
\multiput(0,-.1)(0,.25){5}{\vrule height2pt}
\multiput(2,-.1)(0,.25){5}{\vrule height2pt}
\put(2,1){\line(-2,1){2}}\put(2.02,1){\line(-2,1){2}}
\put(-.4,-.6){\begin{footnotesize}$0$\end{footnotesize}}
\put(.9,-.6){\begin{footnotesize}$1$\end{footnotesize}}
\put(1.9,-.6){\begin{footnotesize}$2$\end{footnotesize}}
\put(1,-.1){\line(0,1){.2}}\put(2,-.1){\line(0,1){.2}}
\put(-.1,1.1){\line(1,0){.2}}\put(-.1,2.1){\line(1,0){.2}}
\put(-.1,3.1){\line(1,0){.2}}\put(-.1,4.1){\line(1,0){.2}}
\put(-.6,.9){\begin{footnotesize}$4$\end{footnotesize}}
\put(-.6,1.9){\begin{footnotesize}$5$\end{footnotesize}}
\put(-.6,2.9){\begin{footnotesize}$6$\end{footnotesize}}
\put(-.6,3.9){\begin{footnotesize}$7$\end{footnotesize}}
\put(-.1,-1.5){\begin{footnotesize}$v_2(b)=2$\end{footnotesize}}

\put(5.85,3.85){$\bullet$}
\put(6.85,2.85){$\bullet$}\put(7.85,.85){$\bullet$}
\put(5,0){\line(1,0){4}}
\put(6,1){\line(0,1){5}}
\put(7,3){\vector(0,1){1}}
\multiput(6,-.1)(0,.25){5}{\vrule height2pt}
\multiput(8,-.1)(0,.25){5}{\vrule height2pt}
\put(8,1){\line(-2,3){2}}\put(8.02,1){\line(-2,3){2}}
\put(5.6,-.6){\begin{footnotesize}$0$\end{footnotesize}}
\put(6.9,-.6){\begin{footnotesize}$1$\end{footnotesize}}
\put(7.9,-.6){\begin{footnotesize}$2$\end{footnotesize}}
\put(7,-.1){\line(0,1){.2}}\put(8,-.1){\line(0,1){.2}}
\put(5.9,1.1){\line(1,0){.2}}\put(5.9,2.1){\line(1,0){.2}}
\put(5.9,3.1){\line(1,0){.2}}\put(5.9,4.1){\line(1,0){.2}}
\put(5.4,.9){\begin{footnotesize}$4$\end{footnotesize}}
\put(5.4,1.9){\begin{footnotesize}$5$\end{footnotesize}}
\put(5.4,2.9){\begin{footnotesize}$6$\end{footnotesize}}
\put(5.4,3.9){\begin{footnotesize}$7$\end{footnotesize}}
\put(4.5,-1.5){\begin{footnotesize}$\begin{array}{c}
v_2(b)=3\\v_2(c+2a+4)\ge4\end{array}$\end{footnotesize}}

\put(11.85,4.85){$\bullet$}\put(11.85,5.85){$\bullet$}
\put(12.85,2.85){$\bullet$}\put(13.85,.85){$\bullet$}
\put(11,0){\line(1,0){4}}
\put(12,1){\line(0,1){5}}
\multiput(12,-.1)(0,.25){5}{\vrule height2pt}
\multiput(14,-.1)(0,.25){5}{\vrule height2pt}
\put(14,1){\line(-1,2){2}}\put(14.02,1){\line(-1,2){2}}
\put(13,3){\line(-1,3){1}}\put(13.02,3){\line(-1,3){1}}
\put(11.6,-.6){\begin{footnotesize}$0$\end{footnotesize}}
\put(12.9,-.6){\begin{footnotesize}$1$\end{footnotesize}}
\put(13.9,-.6){\begin{footnotesize}$2$\end{footnotesize}}
\put(13,-.1){\line(0,1){.2}}\put(14,-.1){\line(0,1){.2}}
\put(11.9,1.1){\line(1,0){.2}}\put(11.9,2.1){\line(1,0){.2}}
\put(11.9,3.1){\line(1,0){.2}}\put(11.9,4.1){\line(1,0){.2}}
\put(11.4,.9){\begin{footnotesize}$4$\end{footnotesize}}
\put(11.4,1.9){\begin{footnotesize}$5$\end{footnotesize}}
\put(11.4,2.9){\begin{footnotesize}$6$\end{footnotesize}}
\put(11.4,3.9){\begin{footnotesize}$7$\end{footnotesize}}
\put(11.4,4.9){\begin{footnotesize}$8$\end{footnotesize}}
\put(10.5,-1.5){\begin{footnotesize}$\begin{array}{c}
v_2(a)\ge3,\,v_2(b)\ge 4\\v_2(c+2a+4)\ge4\end{array}$\end{footnotesize}}

\put(17.85,2.85){$\bullet$}\put(18.85,2.85){$\bullet$}
\put(19.85,.85){$\bullet$}
\put(17,0){\line(1,0){4}}\put(18,1){\line(0,1){5}}
\put(19,3){\vector(0,1){1}}
\multiput(18,-.1)(0,.25){5}{\vrule height2pt}
\multiput(20,-.1)(0,.25){5}{\vrule height2pt}
\put(20,1){\line(-1,1){2}}\put(20.02,1){\line(-1,1){2}}
\put(17.6,-.6){\begin{footnotesize}$0$\end{footnotesize}}
\put(18.9,-.6){\begin{footnotesize}$1$\end{footnotesize}}
\put(19.9,-.6){\begin{footnotesize}$2$\end{footnotesize}}
\put(19,-.1){\line(0,1){.2}}\put(20,-.1){\line(0,1){.2}}
\put(17.9,1.1){\line(1,0){.2}}\put(17.9,2.1){\line(1,0){.2}}
\put(17.9,3.1){\line(1,0){.2}}\put(17.9,4.1){\line(1,0){.2}}
\put(17.4,.9){\begin{footnotesize}$4$\end{footnotesize}}
\put(17.4,1.9){\begin{footnotesize}$5$\end{footnotesize}}
\put(17.4,2.9){\begin{footnotesize}$6$\end{footnotesize}}
\put(17.4,3.9){\begin{footnotesize}$7$\end{footnotesize}}
\put(16.5,-1.5){\begin{footnotesize}$\begin{array}{c}
v_2(b)\ge3\\v_2(c+2a+4)=3\end{array}$\end{footnotesize}}
\end{picture}
\end{center}\be\be
\begin {center}
Figure 9
\end {center}

In the first three cases of Figure 9, $f(x)$ is $(x^2-2)$-regular and we obtain the $2$-integral basis (\ref{basisnu}), with the following values of $Y$ and $\nu$. We also indicate a polynomial $Q(x)$, simpler than $x^2+a+2$,  and for which (\ref{basisnu}) is still a $2$-integral basis. \be

\begin{center}
\begin{small}
\begin{tabular}{|c|c|c|c|}\hline
$a,b,c$&$Y$&$\nu$&$Q(x)$\\\hline
$v_2(b)=2$&$9/2$&$5/4$&$x^2$\\\hline
$v_2(b)=3$, $v_2(c+2a+4)\ge4$&$11/2$&$7/4$&$x^2+2$\\\hline
$v_2(a)\ge3,\,v_2(b)\ge 4,\,v_2(c+2a+4)\ge4$&$6$&$2$&$x^2+2$\\\hline
\end{tabular}
\end{small}
\end{center}\be

In the fourth case of Figure 9, the second order residual polynomial is $(y+1)^2$ and $f(x)$ is $(x^2-2)$-irregular; hence, we must look for other choices for $\phi(x)$.

\subsection{Case $p=2$, $F(x)=g(x)=x^4+4x^3+Ax^2+Bx+C$, $v_2(A)>1$, $v_2(B)>1$, $v_2(C)=2$. }
This case corresponds to row $4$ of Table \ref{T4extra}.
For the different choices of $\phi(x)$ specified in Table \ref{ChoiceofPhiextra}, the polynomial $g(x)$ is $\phi$-regular in second order. The discussion is completely analogous to that of the previous case.

\subsection{Case $p=2$, $F(x)=g(2x)/16=x^4+2mx^3+A'x^2+B'x+C'$, $v_2(A')>1$, $v_2(B')>1$, $v_2(C')=2$. }
This case corresponds to rows $16,17$ of Table \ref{T4extra}, with $A'=A/4$, $B'=B/8$, $C'=C/16$.
The polynomial $F(x)$ is always $\phi$-regular in second order for $\phi(x)=x^2-2$. The $\phi$-development is:
$$F(x)=\phi(x)^2+(2mx+4+A')\phi(x)+(4m+B')x+2A'+C'+4,$$ and the first quotient is $Q(x)=x^2+2mx+2+A'$. We have $\v2((2mx+4+A')\phi(x))=5$ in all cases,
and $\v2((4m+B')x+2A'+C'+4)=5$ or $\ge 6$ according to $v_2(B')\ge 3$ or $v_2(B')=2$. 
Therefore, if $\tau$ is a root of $F(x)$, Theorem \ref{main2} yields a $2$-integral basis 
\begin{equation}\label{basistau}
1,\,\tau,\,Q(\tau)/2^{\lfloor\nu\rfloor},\,Q(\tau)/2^{\lfloor\nu+(1/2)\rfloor},
\end{equation} with the values of $Y$ and $\nu$ indicated in the table below. Actually, we also indicate a polynomial $Q(x)$, simpler than $x^2+2mx+2+A'$,  and for which (\ref{basistau}) is still a $2$-integral basis. \medskip

\begin{center}
\begin{small}
\begin{tabular}{|c|c|c|c|}\hline
$v_2(B')$&$Y$&$\nu$&$Q(x)$\\\hline
$\ge3$&$9/2$&$5/4$&$x^2$\\\hline
$2$&$5$&$3/2$&$x^2+2$\\\hline
\end{tabular}
\end{small}
\end{center}\medskip

In Table \ref{T4extra}, we express the basis (\ref{basistau}) in terms of the root
$\om=2\tau$ of the polynomial $g(x)$.

\begin{center}
\begin{table}
\caption{\footnotesize Choice of $\phi(x)$ for which $f(x)=x^4+ax^2+bx+c$ is $\phi$-regular in second order. Here  $v_2(a)>1$, $v_2(b)>1$, $v_2(c)=2$, and  $u=v_2(b)$, $v=v_2(c-(a^2/4))$, $d=(c-(a^2/4))/2^v\md4$.  }\label{ChoiceofPhi}
\begin{small}
\begin{tabular}{|c|c|c|c|c|c|}\hline
$v_2(a)$&$v_2(b)$&$v_2(2a+c-4)$&$u,v$&$d$&$\phi(x)$\\\hline
&$2$&&&&$x^2-2$\\\hline
&$3$&$3$&&&$x^2-2$\\\hline
&$3$&$\ge4$&&&$x^2-2x-2$\\\hline
$2$&$\ge4$&$\ge4$&&&$x^2-2x-2$\\\hline
$2$&$\ge4$&$3$&$u<v$&&$x^2+\frac a2$\\\hline
$2$&$\ge4$&$3$&$u=v=2w$&&$x^2+\frac a2+2^w$\\\hline
$2$&$\ge4$&$3$&$u=v=2w+1$&&$x^2+2^wx+\frac a2$\\\hline
$2$&$\ge4$&$3$&$u>v=2w$&$-1$&$x^2+\frac a2+2^w$\\\hline
$2$&$\ge4$&$3$&$u>v=2w$&$1$&$x^2+2^wx+\frac a2+2^w$\\\hline
$2$&$\ge4$&$3$&$u>v=2w+1$&$\frac a4$&$x^2+2^wx+\frac a2$\\\hline
$2$&$\ge4$&$3$&$u>v=2w+1$&$-\frac a4$&$x^2+2^wx+\frac a2+2^{w+1}$\\\hline
$\ge3$&$\ge4$&$3$&&&$x^2-2$\\\hline
$\ge3$&$\ge 4$&$4$&&&$x^2-2x-2$\\\hline
$\ge3$&$\ge4$&$\ge5$&&&$x^2-2x+2$\\\hline
\end{tabular}
\end{small}
\end{table}
\end{center}
\begin{center}
\begin{table}
\caption{\footnotesize Choice of $\phi(x)$ for which $g(x)=x^4+4x^3+Ax^2+Bx+C$ is $\phi$-regular in second order, Here $v_2(A)>1$, $v_2(B)>1$, $v_2(C)=2$, and  $u=v_2(B+8-2A)$, $v=v_2(C-((A-4)^2/4))$, $d=(C-((A-4)^2/4))/2^v\md4$.   }\label{ChoiceofPhiextra}
\begin{small}
\begin{tabular}{|c|c|c|c|c|c|}\hline
\mbox{\tiny$v_2(\!A\!)$}&\!\mbox{\tiny$v_2(\!B\!+\!8\!)$}\!&\!\!\mbox{\tiny$v_2(\!2A\!+\!C\!+\!4\!)$}\!\!&$u,v$&$d$&$\phi(x)$\\\hline
&$2$&&&&$x^2-2$\\\hline
&$3$&$\ge4$&&&$x^2-2$\\\hline
$2$&$\ge3$&$3$&&&$x^2+2x-2$\\\hline
$2$&$\ge4$&$\ge5$&&&$x^2-2$\\\hline
$2$&$\ge4$&$4$&&&$x^2+2$\\\hline
$\ge3$&$3$&$3$&&&$x^2+2x-2$\\\hline
$\ge3$&$\ge4$&$3$&$u<v$&&$x^2+2x-2+\frac A2$\\\hline
$\ge3$&$\ge4$&$3$&$u=v=2w$&&$x^2+2x-2+\frac A2+2^w$\\\hline
$\ge3$&$\ge4$&$3$&$u=v=2w+1$&$1+\frac A4$&$x^2+(2+2^w)x-2+\frac A2+2^{w+1}$\\\hline
$\ge3$&$\ge4$&$3$&$u=v=2w+1$&$-1+\frac A4$&$x^2+(2+2^w)x-2+\frac A2$\\\hline
$\ge3$&$\ge4$&$3$&$u>v=2w$&$-1$&$x^2+2x-2+\frac A2+2^w$\\\hline
$\ge3$&$\ge4$&$3$&$u>v=2w$&$1$&$x^2+(2+2^w)x-2+\frac A2+2^w$\\\hline
$\ge3$&$\ge4$&$3$&$u>v=2w+1$&&$x^2+(2+2^w)x-2+\frac A2$\\\hline
$\ge3$&$\ge4$&$\ge4$&&&$x^2-2$\\\hline
\end{tabular}
\end{small}
\end{table}
\end{center}

\begin{thebibliography}{99}
\bibitem[GMN08]{GMN}
Gu\`{a}rdia, J.;  Montes, J.; Nart, E., \emph{Newton polygons of higher order in algebraic number theory}, arXiv:0807.2620v2[math.NT].
\bibitem[GMN09]{GMN2} Gu\`{a}rdia, J.;  Montes, J.; Nart, E., \emph{Higher order Newton polygons and integral bases}, arXiv:0902.3428v1 [math.NT].
\bibitem[Mon99]{m} J. Montes, \emph{Pol\'\i gonos de Newton de orden superior y aplicaciones a\-rit\-m\'eticas}, Tesi Doctoral, Universitat de Barcelona 1999.
\bibitem[Ore23]{ore0} O. Ore, \emph{Zur Theorie der algebraischen K\"orper}, Acta Mathematica 44(1923), pp. 219--314.
\bibitem[Ore28]{ore2} O. Ore, \emph{Newtonsche Polygone in der Theorie der algebraischen K\"orper}, Mathematische Annalen 99(1928), pp. 84--117.
\end{thebibliography}
\end{document}